\newcommand{\sacss}{superalgebraic cartesian sets}
\newcommand{\sacs}{superalgebraic cartesian set}
\newcommand{\Sacss}{Superalgebraic cartesian sets}

\newcommand{\sacp}{superalgebraic cartesian prestack}

\newcommand{\G}{\mathbb{G}}
\newcommand{\e}{\epsilon}
\newcommand{\de}{\delta}
\newcommand{\F}{\mathbb{F}}
\newcommand{\Sect}{\mathcal{O}}
\newcommand{\cO}{\Sect}

\newcommand{\Q}{\mathbb{Q}}
\newcommand{\Z}{\mathbb{Z}}

\newcommand{\R}{\mathbb{R}}
\newcommand{\uRR}{\underline{\mathbb{R}}}

\newcommand{\A}{\mathbb{A}}
\newcommand{\cL}{\mathcal{L}}
\newcommand{\N}{\mathbb{N}}
\newcommand{\M}{\mathbb{M}}

\newcommand{\al}{\alpha}

\newcommand{\lra}[1]{\overset{#1}{\longrightarrow}}

\newcommand{\Prod}[1]{\underset{#1}{\prod}}
\newcommand{\Coprod}[1]{\underset{#1}{\coprod}}
\newcommand{\Colim}[1]{\underset{#1}{\colim}}
\newcommand{\Lim}[1]{\underset{#1}{\lim}}

\newcommand{\sF}{\mathsf{F}}

\newcommand{\cC}{\mathcal{C}}
\newcommand{\be}{\overline{\e}}
\newcommand{\bx}{\overline{x}}

\newcommand{\sA}{\mathsf{sA}} 
\def \mmod{/\mkern-3mu /}

\newcommand{\cV}{\mathcal{V}}

\newcommand{\op}{{\textrm{op}}}

\documentclass[12pt]{amsart}

\usepackage{amssymb,amsfonts,amsthm,amsmath,verbatim}
\usepackage[all]{xy}
\usepackage{tikz}

\usepackage{color}
\usepackage{ifthen}		
\usepackage[linktoc=all]{hyperref}
\hypersetup{
    colorlinks,
    citecolor=black,
    filecolor=black,
    linkcolor=black,
    urlcolor=black
}

\DeclareMathOperator{\uAut}{\underline{Aut}}
\DeclareMathOperator{\End}{End}
\DeclareMathOperator{\uEnd}{\underline{End}}

\DeclareMathOperator{\Bord}{Bord}
\DeclareMathOperator{\BBord}{\mathsf{Bord}}

\DeclareMathOperator{\Hom}{Hom}
\DeclareMathOperator{\uHom}{\underline{Hom}}

\DeclareMathOperator*{\colim}{colim} 

\DeclareMathOperator{\Spec}{Spec}

\DeclareMathOperator{\Mod}{Mod}
\DeclareMathOperator{\MMod}{\mathsf{Mod}}

\DeclareMathOperator{\ob}{ob}

\DeclareMathOperator{\sCart}{sCart}
\DeclareMathOperator{\usCart}{\underline{sCart}}

\DeclareMathOperator{\usMan}{\underline{sMan}}

\DeclareMathOperator{\const}{const}
\DeclareMathOperator{\codis}{codis}

\DeclareMathOperator{\Pic}{Pic}
\DeclareMathOperator{\PPic}{\mathsf{Pic}}

\DeclareMathOperator{\Pre}{Pre}

\DeclareMathOperator{\Set}{Set}

\DeclareMathOperator{\sAlg}{sAlg}
\DeclareMathOperator{\sSet}{sSet}
\DeclareMathOperator{\Vect}{Vect}

\newcommand{\pTFT}[1]  
{
  \ifthenelse{\equal{#1}{}}{  			
		{\mathrm{0|1}\textrm{-}\mathrm{pTFT}} 
	}{									
		{\mathrm{#1}\textrm{-}\mathrm{pTFT} } 
	}
}

\newcommand{\TFT}[1]  
{
  \ifthenelse{\equal{#1}{}}{  			
		{\mathrm{0|1}\textrm{-}\mathrm{TFT}} 
	}{									
		{\mathrm{#1}\textrm{-}\mathrm{TFT} } 
	}
}

\newcommand{\EFT}[1]  
{
  \ifthenelse{\equal{#1}{}}{  			
		{\mathrm{0|1}\textrm{-}\mathrm{EFT}} 
	}{									
		{\mathrm{#1}\textrm{-}\mathrm{EFT} } 
	}
}

\newcommand{\QFT}[1]  
{
  \ifthenelse{\equal{#1}{}}{  			
		{\mathrm{0|1}\textrm{-}\mathrm{QFT}} 
	}{									
		{\mathrm{#1}\textrm{-}\mathrm{QFT} } 
	}
}

\theoremstyle{plain}
\newtheorem{thm}{Theorem}[section]
\newtheorem*{thm*}{Theorem}
\newtheorem{prop}[thm]{Proposition}
\newtheorem{cor}[thm]{Corollary}
\newtheorem{cor-def}[thm]{Corollary/Definition}
\newtheorem{lemma}[thm]{Lemma}

\theoremstyle{definition} 
\newtheorem{definition}[thm]{Definition}
\newtheorem{example}[thm]{Example}
\newtheorem{remark}[thm]{Remark}
\newtheorem{question}[thm]{Question}

\definecolor{CSPcolor}{rgb}{0.0,0.5,0.75}	
\definecolor{NScolor}{rgb}{0.0,0.6,0.1}		

\newcommand{\CSP}[1]{\marginpar{\vspace*{-20pt}\tiny\color{CSPcolor}{ #1}\vspace*{20pt}}}
\newcommand{\NS}[1]{\marginpar{\vspace*{-20pt}\tiny\color{NScolor}{ #1}\vspace*{20pt}}}

\newcommand{\CSPcomm}[1]{{\color{CSPcolor}{#1}}}
\newcommand{\NScomm}[1]{{\color{NScolor}{#1}}}

\begin{document}

	\title{Singular Cohomology  from Supersymmetric Field Theories}

	\author{Christopher Schommer-Pries}
	\address{University of Notre Dame, Department of Mathematics
255 Hurley, Notre Dame, IN 46556}
	\email{cschomme@nd.edu}

	\author{
	Nathaniel Stapleton
	}
	\email{nstapleton@math.mit.edu}
	\thanks{The second author was partially supported by NSF grant DMS-0943787.}
	\address{}

	\date{\today}

\begin{abstract}
	We show that Sullivan's model of rational differential forms on a simplicial set $X$ may be interpreted as a (kind of) $0|1$-dimensional supersymmetric quantum field theory over $X$, and, as a consequence, concordance classes of such theories represent the rational cohomology of $X$. We introduce the notion of {\em \sacss}, a concept of space which should roughly be thought of as a blend of simplicial sets and generalized supermanifolds, but valid over an arbitrary base ring. Every simplicial set gives rise to a \sacs~and so we can formulate the notion of $0|1$-dimensional supersymmetric quantum field theory over $X$, entirely within the language of such spaces. We explore several variations in the kind of field theory and discuss their cohomological interpretations.
	
	Finally, utilizing a theorem of Cartan-Miller, we describe a variant of our theory which is valid over any ring $S$ and allows one to recover the $S$-cohomology $H^*(X;S)$ additively and with multiples of the cup product structure.  	
\end{abstract}

\maketitle


\tableofcontents

\section*{Introduction}

Cohomology theories such as real cohomology, $K$-theory, and cobordism theories have the distinct advantage of a geometric description. They are built out of geometric cochains such as differential forms, vector bundles, or cobordism classes of manifolds. This significantly aids our ability to compute with these theories while also allowing methods from algebraic topology to be used to solve geometric problems.

Chromatic homotopy theory organizes cohomology theories according to their height, which is a measure of the complexity of the theory. Real cohomology and $K$-theory are at heights $0$ and $1$, respectively. The theory of {\em topological modular forms} $TMF$ introduced by Hopkins and Miller is of height $2$, while there are numerous theories, such as Morava $E_n$-theory and $K(n)$-theory, which exist for arbitrary heights $n$. 

In contrast to real cohomology and $K$-theory, there are no known geometric descriptions of these latter theories. In fact, aside from bordism theories (which are manifestly geometric), to our knowledge the only known geometric construction of a cohomology theory of complexity greater than K-theory is via the Baas-Dundas-Richter-Rognes theory of `2-vector bundles' \cite{MR3010546, MR2832571}; it produces $K(ku)$, the algebraic K-theory of topological K-theory, a theory of telescopic complexity two.  

Nevertheless, several years ago the enticing idea was put forward that quantum field theories could provide some of the best candidates for geometric cochains for higher height cohomology theories. This idea was pioneered by Graeme Segal \cite{MR992209} who proposed to use 2-dimensional conformal field theories to give geometric cocycles for elliptic cohomology. This idea has been further developed in the work of Stolz-Teichner \cite{MR2079378, Stolz-Teichner-Survery2}.

While the primary goal of the Stolz-Teichner program has  been to use quantum field theories to construct a geometric model of $TMF$, a goal which has not yet been fully realized, as an offshoot they have been very successful in constructing new geometric models of K-theory and de Rham cohomology based entirely on the formalism of quantum field theory. See \cite{MR2763085} for the latter case. 

There are two categories which go into the Atiyah-Segal formulation of quantum field theories.
\begin{itemize}
	\item A symmetric monoidal category (or more generally $n$-category) $\Bord$ of bordisms. Here the objects are manifolds (say of dimension $d-1$) and the morphisms are isomorphism classes of bordisms between these. In the context relevant to cohomology theories these manifolds will typically be equipped with some geometric structure such as metrics or conformal structures, though the purely topological case is also of interest. 
	\item A target symmetric monoidal category $\cV$. This is often the category $\Vect$ of vector spaces (or Hilbert spaces). In higher categorical contexts a suitable higher categorical analog of vector spaces should be used. 
\end{itemize}
A quantum field theory is then defined to be a symmetric monoidal functor:
\begin{equation*}
	Z: \Bord \to \cV.
\end{equation*}

When there is geometry involved the set of all choices of that geometry (on a given bordism) will form a kind of `space', and our quantum field theory should restrict to give a function (continuous, smooth, holomorphic, etc.) on that space. In certain degenerate cases these `spaces' will actually themselves be represented by manifolds, but more generally we will need to use `generalized manifolds' (i.e. concrete sheaves) or stacks. 

It is important that quantum field theories respect this structure. One way to accomplish this (following \cite[\S2]{Stolz-Teichner-Survery2}) is to regard $\Bord$ as an internal category, internal to stacks or generalized manifolds. The target category $\cV$ will be of the same kind and our field theory is required to be an internal functor. 

There are several other key ideas which play a role in the Stolz-Teichner program. One of them is the use of {\em supersymmetric} quantum field theories. The theory of supermanifolds and resulting supergeometry are used extensively in their work.  Another key idea is that it is possible to form {\em twisted field theories}, and in particular field theories of a fixed {\em degree} $n \in \Z$. A third ingredient is that it is possible to consider field theories {\em over a (super) manifold} $X$, in which the relevant cobordisms are equipped with maps to $X$. This will be (contravariantly) functorial in $X$ and hence one obtains a series of (pointed) presheaves:
\begin{equation*}
	X \mapsto \QFT{}^n(X)
\end{equation*} 
Here $\QFT{}^n(X)$ denotes the set of isomorphism classes of degree $n$ quantum field theories over $X$.
By varying the dimension of the bordisms, the geometry, and the target category one obtains a plethora of varieties of quantum field theories. Its flexibility is part of the appeal of this subject.

Two quantum field theories over $X$ are defined to be {\em concordant} if there exists a quantum field theory over $X \times \R$ which restricts to the two given fields theories on $X \times \{i\}$, $i=0,1$. Concordance induces an equivalence relation, and we denote the set of concordance classes of quantum field theories over $X$ by $\QFT{}^n[X]$. It is automatically homotopy invariant. In very favorable situations this construction yields a cohomology theory;
 this is the case for de Rham cohomology \cite{MR2763085}, K-theory \cite{stolz-teichner-unpublished}, Tate K-theory \cite{Cheung:2008aa}, and complexified $TMF$ \cite{Berwick-Evans:2013aa}.
 In the current work we build on these ideas. We were particularly influenced by the results of  Hohnhold-Kreck-Stolz-Teichner \cite{MR2763085}. 

The first major departure from previous results is a move away from (generalized) supermanifolds. In section \ref{sec:sacs} we introduce the notion of {\em \sacss}. 
One way to view manifolds, and also more exotic `generalized manifolds', is as certain sheaves on the category of smooth cartesian spaces, i.e. the category with objects $\R^n$ for $n \in \N_{\geq 0}$ and morphisms $\hom(\R^n, \R^m)$ the set of smooth maps from $\R^n$ to $\R^m$. Similarly generalized supermanifolds may be viewed as certain sheaves on the category of smooth supercartesian spaces $\R^{n|q}$. 

\Sacss~are defined analogously but with the following changes:
\begin{itemize}
	\item We drop the sheaf requirement, allowing ourselves to consider arbitrary presheaves (and indeed arbitrary prestacks);
	\item Instead of all smooth maps between $\R^{n|q}$ and $\R^{m|p}$, we restrict to functions which are polynomials in the standard coordinates;
	\item We allows these polynomials to be defined over an arbitrary base ring. 
\end{itemize}
 Consequently we find it more appropriate to denote the representable superalgebraic cartesian sets as $\A^{n|q}$. The term `\sacs' is supposed to remind us that this notion of space is based on the polynomial algebra over an arbitrary ring, while also being evocative of the term `simplicial set'. Indeed any simplicial set has an {\em algebraic realization} as a \sacs, and any \sacs~has a corresponding singular simplicial set (see Section~\ref{sec:sacs}). They also have several aspects reminiscent of schemes in algebraic geometry, though the theory of \sacss~is more simplistic. Everything we do is functorial in the base ring. 
 
Given this new notion of space, we may then mimic the usual definition of quantum field theory. In this paper we will focus on the simplest species of supersymmetric quantum field theories, those of superdimension $0|1$. The bordisms in this case consist of finite disjoint unions of the representable {\em superpoint } $\A^{0|1}$. 

A second departure from previous work is that instead of working over a supermanifold, we define these quantum field theories over an arbitrary simplicial set. We consider a variety of geometries on the superpoint, each of which gives rise to a notion of supersymmetric ${0|1}$-dimensional quantum field theory. We classify the possible global twists for these theories. In each case there is always a {\em degree $n$ twist} where $n \in \N$ now takes values in the natural numbers. 

When the base ring is the field $\Q$ of rational numbers, the supersymmetric ${0|1}$-dimensional quantum field theories over a simplcial set~$X$ have a familiar interpretation. They coincide precisely with Sullivan's model of rational polynomial differential forms on $X$ \cite{Sull}. More precisely, the most interesting geometries we consider are: fully-rigid, Euclidean, and topological (no geometry). In these cases we obtain the following result:

\begin{thm*}\label{thm:mainthm}
	Let $R$ be a rational algebra, and consider the category of \sacss~defined over $R$. Let $X$ be a simplicial set, regarded as a \sacs. Then:
	\begin{enumerate}
		\item For each of the following geometries the set of supersymmetric ${0|1}$-dimensional quantum field theories of degree $n$ over $X$ may be identified as: 
		\begin{enumerate}
			\item (topological) closed degree $n$ polynomial forms over $R$
			\begin{equation*}		
				\TFT{}^n(X) \cong \Omega^n_{R; cl}(X);
			\end{equation*}
			\item (Euclidean) closed periodic polynomial forms over $R$
			\begin{equation*}
				\EFT{}^n(X) = \begin{cases}
					\Omega^\textrm{ev}_{R; cl}(X) & n \textrm{ even} \\
					\Omega^\textrm{odd}_{R; cl}(X) & n \textrm{ odd}
				\end{cases}
			\end{equation*}
			\item (fully-rigid) all polynomial forms over $R$ 
\[
\QFT{}_\textrm{f-r}(X) \cong \Omega^*_{R}(X).
\]
		\end{enumerate}
		\item For each of the following geometries the set of concordance classes of supersymmetric ${0|1}$-dimensional quantum field theories of degree $n$ over $X$ may be identified as: 
		\begin{enumerate}
			\item (topological) $\TFT{}^n[X] \cong HR^n(X)$  degree $n$ $R$-cohomology;
			\item (Euclidean) $\EFT{}^n[X]\cong PHR^n(X)$ periodic $R$-cohomology.
		\end{enumerate}
	\end{enumerate}
Moreover in the case of fully-rigid geometry the natural symmetries of the the supersymmetric quantum field theory recover the commutative differential graded algebra structure on  $\Omega^*_{R}(X)$.	  
\end{thm*}

For rings $S$ which are \emph{not} rational we have a useful variant of the above theory, inspired by the theorems of Cartan-Miller \cite{Cartan1976, miller1978derham}. In this variant the base ring is taken to be $R = \Gamma_S(t)$, the free divided powers $S$-algebra on a single generator $t$. The functions on $\A^{n|q}$ are no longer the polynomial algebra over $R$, but are further enhanced with divided powers. As observed by Cartan and Miller, this is enough to define an integration map from forms to simplcial cochains over $R$, and a slightly weaker version of Sullivan's theorem holds as well. In the language of field theories we have

\begin{thm*}
	Let $S$ be any ring and consider the category of \sacss~with divided powers defined over $R = \Gamma_S(t)$. Let $X$ be a simplicial set, regarded as a one of these spaces. Then `integration' gives a natural isomorphism 
	\begin{equation*}
		TFT^*[X] \cong  \oplus_p H^p(X; \Gamma^{\geq p}_S(t)) \subseteq H^p(X; \Gamma_S(t))
	\end{equation*}
	between the concordance classes of supersymmetric ${0|1}$-dimensional topological quantum field theories over $X$ and the specified subring of the cohomology   $H^p(X; \Gamma_S(t))$.
\end{thm*}

We refer the reader to Section~\ref{sec:cohomology} for full details, but note that since $\Gamma_S(t)$ is flat over $S$,
\begin{equation*}
	H^p(X; \Gamma_S(t)) \cong H^p(X; S) \otimes_S \Gamma_S(t) \cong \oplus_\infty H^p(X; S)
\end{equation*}
is simply a direct sum of countably many copies of $H^*(X; S)$. In particular topological quantum field theories encode the $S$-cohomology of $X$ for any ring $S$.

\subsection*{Further motivations}

One tool that aids in the study of higher height cohomology theories is a form of character theory \cite{hkr, tgcm}. It provides a character map that approximates high height cohomology theories by a form of rational cohomology. The form of rational cohomology has coefficients that are a ring extension of the rationalization of the coefficients of the high height cohomology theory. These rings are often algebras over the p-adic rationals.

Many features of these character maps are reminiscent of {\em dimensional reduction} maps between field theories. In fact there is a quantum field theoretic interpretation of the (Bismut) Chern character map which arises precisely as a dimensional reduction \cite{Han:2007aa}. This geometric construction yields a character map from K-theory taking values in periodic de Rham cohomology. 

Periodic de Rham cohomology cannot be a suitable target for the higher height character maps that take place at a prime $p$. This is essentially because there is no (interesting) map from the real numbers $\R$ to the $p$-adic rationals $\Q_p$. For example the $p$-adic Chern character may be obtained as the completion of the ordinary Chern character, but only once it is factored through periodic {\em rational} cohomology. 

This project grew out of a desire to explore the relationship between higher character theory and quantum field theory, which remains an ongoing project. This paper achieves a crucial first step, which is to construct a geometric and quantum field theoretic construction of the cohomology theories which serve as targets of these higher character maps.



\subsection*[SUSY field theories and de Rham cohomology]{Review of the literature}

In this section we give a rapid summary of the work of Hohnhold-Kreck-Stolz-Teichner \cite{MR2763085} relating smooth differential forms to $0|1$-dimensional supersymmetric quantum field theories. This material serves as a conceptual blueprint for the theory developed in later sections.

The Atiyah-Segal axioms define quantum field theories as symmetric monoidal functors from a bordism category, $\Bord$, to a target symmetric monoidal category $\cV$, such as the category of Hilbert spaces. This can be generalized in many ways. One important way is that the theory can be made to satisfy an enhanced form of locality by replacing ordinary 1-categories with $d$-categories. Thus $\Bord^d$ is to be a symmetric monoidal $d$-category. The objects are to be 0-dimensional bordisms, the morphisms are 1-dimensional bordisms, the 2-morphisms are 2-dimensional bordisms between bordisms, etc. all the way up through dimension $d$. A fully local quantum field theory is then a symmetric monoidal functor from this symmetric monoidal $d$-category $\Bord^d$ to another target symmetric monoidal $d$-category $\cV$. 

The work of HKST \cite{MR2763085} considers a degenerate case where $d=0$. Thus the bordism $n$-category becomes a symmetric monoidal 0-category. That is to say it becomes a commutative monoid. The target category $\cV$ will also be replaced by a commutative monoid, and field theories become commutative monoid homomorphisms. 

Furthermore, these theories are supersymmetric. We refer the reader to \cite{MR1701597} for the necessary background material on supermanifolds. In this case the bordisms are (closed) compact\footnote{A supermanifold will be considered compact if the underlying reduced manifold is compact.} supermanifolds of dimension~$0|1$. Each such bordism is a finite disjoint union of copies of the {\em superpoint} $\R^{0|1}$. Each of these bordisms could also be equipped with some kind of geometry which defines the kind of quantum field theory, and HKST consider two possibilities: topological and Euclidean. For simplicity in the remainder of this section we will focus on the topological case (with no geometry).
However even in the topological case we will equip the bordisms with a further structure.
For each supermanifold $X$, HKST consider a category of bordisms {\em over $X$} where each $0|1$-dimensional bordism is equipped with the structure of a map to $X$.  

Finally, as mentioned in the introduction, the bordism category  $\Bord_X^{0|1}$ is constructed internally to the category of stacks on the Grothendieck site of supermanifolds. In this case this means that $\Bord_X^{0|1}$ is a commutative monoid object in stacks on supermanifolds. The easiest way to describe this object is using the {\em $S$-point formalism} (which is essentially the same as the fibered category approach to stacks (See \cite{MR2223406} for an excellent introdution to fibered categories and stacks)). 

Given an arbitrary test supermanifold $S$, we will describe the symmetric monoidal groupoid of maps from $S$ into  $\Bord_X^{0|1}$. Its objects consist of {\em $S$-families} of $0|1$-dimensional bordisms equipped with a map to $X$. More precisely the objects consist of a pair $(E,f)$ where $E$ is a bundle of ${0|1}$-dimensional supermanifolds over $S$ equipped with a map $f:E \to X$ from the total space to $X$. There is an obvious notion of automorphism making this a groupoid, and the symmetric monoidal structure is given by fiberwise disjoint union over $S$.

The definition of quantum field theory is incomplete without the target category, which will be a 0-category analog of the category of vector spaces. In this case HKST use the categorical looping of the category of vector spaces, which is the representable stack $\R$ (endomorphisms of the unit vector space). To indicate that we are thinking of $\R$ as a representable stack we will write it with an underline $\uRR$. Multiplication makes $\uRR$ into a commutative monoid object in supermanifolds (and hence also in stacks on supermanifolds), and $0|1$-dimensional topological quantum field theories over $X$ are the defined \cite[Def.~5.1]{MR2763085} to be the set
\begin{equation*}
	\TFT{}(X) = \Hom(\Bord_X^{0|1}, \uRR)
\end{equation*}
of homomorphisms of commutative monoids in stacks over supermanifolds.

As a commutative monoid $\Bord_X^{0|1}$ is freely generated by the stack quotient 
\begin{equation*}
	\usMan(\R^{0|1}, X) \mmod \uAut(\R^{0|1}).
\end{equation*}
Here $\usMan(\R^{0|1}, X)$ is the internal mapping object from $\R^{0|1}$ into $X$, and $\uAut(\R^{0|1})$ is the internal automorphism object; an $S$-point of $\usMan(\R^{0|1}, X)$ is a map $S \times \R^{0|1} \to X$ and the effect of taking the stack quotient is that, locally in $S$, we may glue these trivial $\R^{0|1}$-bundles together to form non-trivial bundles. 
Both $\usMan(\R^{0|1}, X)$ and $\uAut(\R^{0|1})$ turn out to be representable by supermanifolds. The former is given by
\begin{equation*}
	\usMan(\R^{0|1}, X) \cong \pi TX
\end{equation*}
and the latter is a super Lie group $\R^{\times} \ltimes \R^{0|1}$ \cite[Prop.~3.1 and Lma.~3.5]{MR2763085}. The supermanifold $\pi TX$ has the surprising property that its algebra of functions is the superalgebra of differential forms on $X$.  

With this description of the bordism category, it is straightforward to calculate the topological field theories explicitly. Since the commutative monoid $\Bord_X^{0|1}$ is freely generated by $\pi TX \mmod \uAut(\R^{0|1})$, the commutative monoid homomorphisms from $\Bord_X^{0|1}$ into $\uRR$ are exactly the same as the maps from $\pi TX \mmod \uAut(\R^{0|1})$ into $\uRR$. These in turn may be identified with the even functions on $\pi TX$ (i.e. even differential forms) which are $\uAut(\R^{0|1})$-invariant. These are exactly the locally constant functions on $X$ (i.e. the closed degree zero differential forms on $X$) \cite[Prop.~5.5]{MR2763085}.   

{\em Twisted} quantum field theories generalize the quantum field theories just described. For a given geometry there is a symmetric monoidal category of twists, and for each twist, $\tau$, there is a corresponding notion of $\tau$-twisted quantum field theory over $X$. Let $\QFT{}^\tau(X)$ denote the set of these. In the above situation there is a natural family of {\em degree $n$ twists} parametrized by the integers (with the untwisted case corresponding to degree zero). 

A similar calculation to the above \cite[Prop.~6.3]{MR2763085} yields
\begin{equation*}
	\TFT{}^n(X) \cong \Omega^n_\textrm{cl}(X),
\end{equation*}    
that is $0|1$-dimensional degree $n$ supersymmetric topological field theories over $X$ are in natural bijection with closed smooth differential forms of degree $n$ on $X$. As a consequence concordance classes of these degree $n$ topological field theories are in bijection with de Rham cohomology classes on $X$. More generally the category of twists depends on the manifold $X$ and twisted quantum field theories give a model of twisted cohomology \cite{SST}.  
We will make the definition of twist more precise in our specific context in Section~\ref{sec:SQFT2}. 

\subsection*{Outline of the paper}



In Section \ref{sec:sacs} we define the category of \sacss. The category is a presheaf topos and we develop basic properties of the category from that perspective. The category has a
distinguished supercommutative algebra object $\cO$. In Section \ref{sec:scommalg-in-sacs} we study the Picard category of invertible modules for $\cO$. This is important when studying twisted field theories in Section \ref{sec:twists}. 

In Sections  \ref{sec:tiny} and \ref{sec:endos} we study the mapping space from the superpoint into a \sacs~and show that under certain conditions the action of the endomorphisms of the superpoint on the mapping space produces a cdga structure. In Section  
\ref{sec:forms} we examine more closely the case where the \sacs~comes from a simplicial set $X$, we show that the ring of functions on this mapping \sacs~is precisely Sullivans rational differential forms on $X$, and that the endomorphisms of the superpoint reproduce the grading and differential on Sullivan's rational differential forms. 

Section \ref{sec:geometries} explores the structure induced by submonoids of the endomorphisms of the superpoint. These are called geometries. In Section \ref{sec:SQFT2} we define and study $0|1$-dimensional supersymmetric quantum field theories in the context of \sacss. In analogy to the smooth setting, we define a bordism (0-)category over an arbitrary \sacs~$X$. The bordisms in this case consist of finite disjoint unions of copies of the superpoint $\A^{0|1}$ and they are equipped with maps to the \sacs~$X$. For each geometry we describe the collection of $0|1$-dimensional supersymmetric quantum field theories over $X$ in terms of Sullivan's rational differential forms. In Section \ref{sec:twists} we define twisted field theories and describe the twisted field theories in terms of rational differential forms. Various natural notions of concordance are defined in Section \ref{sec:concordance} and we show that they are all equivalent. This gives the main theorem. In Section~\ref{sec:cohomology} we describe a variant which recovers the Cartan-Miller theory of divided power differential forms over an arbitrary (possibly non-rational) ring.

\subsection*{Acknowledgments}
We would like to thank Peter Teichner for several useful conversations and for suggesting that we look at Sullivan's work as a geometric model of rational cohomology. We would also like to thank Martin Olbermann for highlighting an important difference between the linear and non-linear twists. We would like to thank Gerd Laures for pointing us to the work of Cartan. We would like to thank the Max Planck Institute for Mathematics in Bonn for their generous hospitality; this work was carried out at the MPIM.

\section{Superalgebraic Cartesian Sets} \label{sec:sacs}

By a $\Z/2$-graded commutative ring we will mean a $\Z/2$-graded ring which satisfies the commutativity condition 
$x y = (-1)^{|x| \cdot |y|} y x$ for all homogeneous elements $x$ and $y$. Let $\sAlg$ be the category of $\Z/2$-graded commutative rings and grading preserving homomorphisms. We will refer to objects in this category as supercommutative rings. 

Given a commutative ring $R$ we may form a supercommutative ring denoted $R[x_1,\ldots,x_n,\e_1,\ldots,\e_q]$, where we implicitly assume that the variables $x_i$ refer to even generators and the variables $\e_j$ refer to odd generators. This ring is the tensor product over $R$ of a polynomial ring on even variables $x_1, \ldots, x_n$ with an exterior algebra on odd variables $\e_1, \ldots, e_q$. 

To be completely explicit, this means that $R[x_1,\ldots,x_n,\e_1,\ldots,\e_q]$ is presented as 
\begin{equation*}
	R[x_1, \ldots, x_n]\langle \e_1,\ldots,\e_q \rangle / (\e_i^2, \e_i \e_j + \e_j \e_i, i, j = 1, \ldots,q),
\end{equation*}
the quotient of the free polynomial ring on the generators $x_i$ adjoined non-commutative generators $\e_j$ by the two-sided ideal generated by the squares $\e_j^2$ and the supercommutators $\e_i \e_j + \e_j \e_i$. Since this ideal is generated by even degree elements, the result is still a $\Z/2$-graded ring and is manifestly graded commutative. This formulation is important when 2 is not a unit (such as in characteristic 2) where the relation $\e^2 = 0$ for odd elements $\e$ \emph{does not follow} from the super commutation relation alone. If 2 is invertible, then the square terms are already contained in the two-sided ideal generated by the supercommutators $\e_i \e_j + \e_j \e_i$. 

\begin{remark}
Here and throughout the paper the degree of anything called $\e$ or $\delta$ will be odd. Thus these are square zero elements of the supercommutative ring. 
\end{remark}

In this section we introduce \sacss. These are a species of space which are a primordial mixture of the concepts of supermanifold, (super)scheme, and simplicial set. While everything we will explain in this section is super (i.e. $\Z/2$-graded commutative), one could just as well form an ungraded analogue called algebraic cartesian sets.

\begin{definition}
Fix a commutative ring $R$. The superalgebraic cartesian category $\sA$ has objects $\A^{n|q}_{R}$ for $n,q \in \mathbb{N}$ and morphisms the polynomial maps 
\[
\sA(\A^{n|q}_{R}, \A^{m|p}_{R}) = \sAlg^{\text{op}}(R[x_1,\ldots,x_n,\e_1,\ldots,\e_q],R[x_1,\ldots,x_m,\e_1,\ldots,\e_p]).
\]
\end{definition}

\noindent Hence $\sA$ is a full subcategory of the opposite of the category of supercommutative $R$-algebras. 

\begin{definition}
The category of \sacss~is the category of presheaves $\sCart := \Pre(\sA)$. A \sacs~is an object of $\sCart$.
\end{definition}

\begin{example}
	We will often abuse notation and write $\A^{n|q}$ instead of $\A^{n|q}_{R}$ for the {\em representable}	\sacss, however, everything that we do will be functorial in the ring $R$.  
	Note that $\A^{n|q} \cong (\A^1)^{n}\times(\A^{0|1})^{q}$.
The {\em superpoint} is the \sacs~$\A^{0|1}$.
\end{example}

\subsection[as a presheaf topos]{Superalgebraic cartesian sets as a presheaf topos} \label{subsec:topos}

The category of \sacss~is, by definition, a presheaf topos and consequently it enjoys the nicest possible categorical properties. 

\begin{example} \label{ex:innerhom}
The category $\sCart$ is cartesian closed.
 The categorical product of two \sacss~$X$ and $Y$ is computed pointwise, and for each \sacs~$X$,
 the right adjoint to $X \times (-)$ is given by the internal mapping functor $\underline{\sCart}(X, -)$.  
 The internal mapping \sacs~is given as the presheaf
\[
\underline{\sCart}(X,Y): \sCart \lra{} \Set
\]
mapping
\[
\A^{n|q} \mapsto \sCart(\A^{n|q} \times X, Y).
\]
\end{example}

The category of \sacss~is complete and cocomplete with both limits and colimits computed pointwise
\[
(\colim X_\alpha)(\A^{m|p}) = \colim (X_\alpha(\A^{m|p})) 
\]
and 
\[
(\lim X_\alpha)(\A^{m|p}) = \lim (X_\alpha(\A^{m|p})). 
\]
As a topos, \sacss~are also a context in which to carryout mathematics. We can almost effortlessly study the theories of groups, monoids, commutative rings, modules, categories, and even supercommutative rings, internally to \sacss. 

\begin{example}\label{example:the_ring_opbject_O}
		There is an important supercommutative algebra object $\cO \in \sCart$. As a \sacs~we have $\cO = \A^{1|1}$. Addition is given by 	
	\[
	R[x,\e] \lra{} R[x_1,x_2,\e_1,\e_2]:(x \mapsto x_1+x_2, \e \mapsto \e_1 + \e_2)
	\]
	and multiplication is given by
	\[
	R[x,\e] \lra{} R[x_1,x_2,\e_1,\e_2]: (x \mapsto x_1x_2+\e_1\e_2, \e \mapsto x_1\e_2+x_2\e_1),
	\]
	where we have used the embedding of $\sA$ into $\sAlg^{\text{op}}$ to write down these maps.
\end{example}

Every topos has a {\em global sections} functor $\Gamma$ which is given by evaluation on the terminal object. In the language of topos theory this is a geometric morphism to the terminal topos, the category of sets. In the case at hand, we have even more structure. Since the category $\sA$ has all finite products the category of \sacss~is a {\em cohesive topos} \cite{MR2125786, MR2369017} (just like the category of simplicial sets). 
This means that we have a series of adjunctions:
\begin{equation*}
\pi_0 \dashv \const \dashv \Gamma \dashv \codis
\end{equation*}
and moreover the functor $\pi_0$ commutes with finite products. In more detail these functors are given by: 
\begin{align*}
	\codis: \Set & \to \sCart \\
	S & \mapsto (\A^{m|p} \mapsto S^{R^{\times m}} = S^{\Gamma(\A^{m|p})}) \\
	\Gamma: \sCart & \to \Set \\
	X & \mapsto X(\A^0) \\
	\const: \Set & \to \sCart \\
	S & \mapsto \coprod_S \A^0 \\
	\pi_0 : \sCart & \to \Set \\
	X & \mapsto \colim_{\sA^\op} X
\end{align*}
The functor $\pi_0$ sends a \sacs, viewed as diagram of sets indexed on $\sA^\op$, to its colimit. The functor $\Gamma$ evaluates a \sacs~on the terminal object $\A^0$. The functor $\const$ sends a set to the constant presheaf on that set, and $\codis$ sends a set to the {\em codiscrete} \sacs~on that set. 

These functors allow us to pass back and forth between set based mathematical concepts and those same concepts developed internally to \sacss. For example every ring object in \sacss~has, via the functor $\Gamma$, an underlying ordinary ring. For example $\Gamma(\cO) = R$ is our chosen base ring. Similarly every ordinary ring may be augmented, via the functor $\const$, to a ring object internal to \sacss. The counit map
\begin{equation*}
	\const(R) \to \cO
\end{equation*}
is automatically a map of ring objects. These observations will be used in Section~\ref{sec:scommalg-in-sacs}. 

As we mentioned above, all of these considerations are functorial in the base ring $R$. A ring homomorphism $R' \to R$ induces a functor $\sA_{R'} \to \sA_R$ and hence gives rise to a geometric morphism of topoi 
\begin{equation*}
	f^*: \sCart_{R'} \leftrightarrows \sCart_R: f_*
\end{equation*}
where the {\em restriction of scalars} $f_*$ is given by precomposition with $\sA_{R'}^\op \to \sA_R^\op$. In particular since this is a geometric morphism of topoi the left-adjoint, which is given by left Kan extension along the Yoneda embedding, commutes with finite limits. Moreover this morphism of topoi is {\em local}, that is the functor $f_*$ admits a further right adjoint $f^!: \sCart_{R'}  \to \sCart_R$.


\subsection[and superalgebras]{Superalgebraic cartesian sets and superalgebras} \label{subsec:superalg}

Superalgebraic cartesian sets have a close connection to superalgebras and superschemes. 
The category $\sA$ is the multisorted Lawvere theory\footnote{In fact it is a {\em super Fermat theory} \cite{Carchedi:2012ab}.} for supercommutative $R$-algebras,  
which means that supercommutative $R$-algebras in any category $\cC$ with finite products are the same as product preserving functors $\sA \to \cC$. The {\em generic object} of $\sA$ is the supercommutative $R$-algebra $\cO$ from Example~\ref{example:the_ring_opbject_O}. 

\begin{example}
	The Yoneda embedding $\sA \to \sCart$ preserves products and corresponds to the supercommutative $R$-algebra object $\cO$ in \sacss~as in Example~\ref{example:the_ring_opbject_O}. 
\end{example}


Recall that $\sA$ is a full subcategory of $\sAlg^{\text{op}}$. The embedding of $\sA$ into $\sAlg^{\text{op}}$ is via the functor $\cO(-) = \sCart(-, \cO)$. This formula extends the functor $\cO$ to all of $\sCart$, and for a \sacs~$X$ we will refer to $\cO(X)$ as the {\em ring of global functions} on $X$.

\begin{example} \label{ex:superspec}
	The functor $\cO: \sCart \to \sAlg^\op $ from \sacss~to the opposite category of supercommutative $R$-algebras is easily seen to commute with colimits.  It follows that it is given by left Kan extension of its restriction to $\sA$ along the Yoneda embedding. 
	\begin{center}
	\begin{tikzpicture}
			\node (LT) at (0, 1.5) {$\sA$};
			\node (LB) at (0, 0) {$\sCart$};
			\node (RT) at (4, 1.5) {$\sAlg^{\text{op}}$};
			\draw [->] (LT) -- node [left] {$y$} (LB);
			\draw [->] (LT) -- node [above] {$\cO$} (RT);
			\draw [->] (LB) -- node [above left] {$\cO$} (RT);
			\draw [transform canvas={yshift=-1ex},->] (RT) -- node [below right] {$\cO^*$} (LB);
			
	\end{tikzpicture}
	\end{center}
We obtain an adjunction:
\begin{equation*}
	\cO: \sCart \rightleftarrows \sAlg^{\text{op}}: \cO^*.
\end{equation*}		
	The right adjoint $\cO^*$ is the functor sending a supercommutative algebra $A$ to the \sacs~defined via
	\begin{equation*}
		\sCart(\A^{n|q}, \cO^*(A)) \cong \sAlg^{\text{op}}(\cO(\A^{n|q}), A) = \sAlg(A, \cO(\A^{n|q})).
	\end{equation*} 
	Thus every supercommutative algebra gives rise to a \sacs. 
\end{example}



\begin{example} \label{ex:Omega}
We define a \sacs~called $\Omega$ (purposefully similar to Sullivan's $\Omega^{*}_{\bullet}$ introduced in Section~\ref{sec:forms}) which sends $\A^{n|q}$ to the supercommutative ring $\Sect(\underline{\sCart}(\A^{0|1},\A^{n|q}))$. Thus $\Omega$ is another supercommutative ring object in \sacss. It is an algebra over the supercommutative ring $\cO$, and we will see in Section~\ref{sec:endos} that $\Omega(\A^{n|q})$ is isomorphic to the ring of K\"ahler differential forms on $\cO(\A^{n|q})$. 
\end{example}

\subsection[and simplicial sets]{Superalgebraic cartesian sets and simplicial sets}\label{sec:SACS_simplicial}

Let $\Delta$ be the category of combinatorial simplices (i.e. the category of finite non-empty totally ordered sets and order preserving maps). There is an important faithful functor (which factors through the category of (non-super) algebraic cartesian sets)
\[
i: \Delta \lra{} \sA.
\]
The functor $i$ sends $[n]$ to $\A^{n} = \A^{n|0}$ and we use the isomorphism 
\[
\Sect(\A^n) = R[x_1, \ldots, x_n] \cong R[x_0,\ldots,x_n]/(\Sigma_i x_i - 1)
\]
to see the simplicial maps and identities. The $\A^{n|0}$ may be viewed as {\em extended simplices}. 

\begin{example}\label{example:sSetasSACS}
Let $\sSet = \Pre(\Delta)$ be the category of simplicial sets. We apologize for the use of the letter ``s" for both simplicial and super. Given a simplicial set $X$, we can form a \sacs~by left Kan extension. We have the following diagram
\begin{center}
\begin{tikzpicture}
		\node (LT) at (0, 1.5) {$\Delta$};
		\node (LB) at (0, 0) {$\sSet$};
		\node (MT) at (2, 1.5) {$\sA$};
		\node (RT) at (4, 1.5) {$\sCart$};
		\draw [->] (LT) -- node [left] {$y$} (LB);
		\draw [->] (LT) -- node [above] {$i$} (MT);
		\draw [->] (MT) -- node [above] {$y$} (RT);
		\draw [->] (LB) -- node [above left] {$i_!$} (RT);
		\draw [transform canvas={yshift=-1ex},->] (RT) -- node [below right] {$i^*$} (LB);
\end{tikzpicture}
\end{center}
and the \sacs~associated to $X$ is $i_{!}X$, the left Kan extension along the Yoneda embedding. We will call this the {\em algebraic realization} of $X$, in analogy with the geometric realization. 
This fits into an adjunction with the restriction functor $i^*$ that brings a \sacs~to its underlying simplicial set. 
\begin{equation*}
	i_!: \sSet \rightleftarrows \sCart: i^*
\end{equation*}
Given a \sacs~$Y$ and a simplicial set $X$, there is a natural isomorphism
\[
\sSet(X, i^* Y) \cong \sCart(i_! X, Y). 
\]
As a left adjoint, $i_!$ commutes with colimits.

Furthermore, $i^*$ also commutes with colimits, hence it admits a further right adjoint $i_*$, given by right Kan extension. The triple $(i_!, i^*, i_*)$ constitutes an {\em essential morphism of topoi} \cite{MR2369017} from $\sSet$ to $\sCart$. 
\end{example}

\begin{prop}
	Recall the functor $\pi_0: \sCart \to \Set$ introduced previously. We have the equality $\pi_0 \cong \pi_0 \circ i^*$, in other words the functor $\pi_0$ applied to a \sacs~may be computed as the path components of the underlying simplical set. Similarly $\pi_0 \cong \pi_0 \circ i_!$, the path components of a simplical set may be computed as the value of $\pi_0$ applied to its algebraic realization. 
\end{prop}

\begin{proof}
	Recall that $\pi_0 X = \colim_{\sA^\op}X$ and that $\pi_0 i^*(X) = \colim_{\Delta^\op} X \circ i$. Thus one way to see this is to show directly that $\Delta^\op$ is cofinal in $\sA^\op$. Alternatively first observe that $i_*$ sends discrete simplicial sets $\coprod_S \Delta^0$ to constant \sacss~$\coprod_S \A^0 = \const(S)$. This follows formally from the observation that $i^* \A^{m|p}$ is a connected simplical set for each $m|p$. From this the above proposition follows immediately since for any set $S$ and any \sacs~$X$ we have
	\begin{align*}
		\Set( \pi_0 X, S) &\cong \sCart(X, \const(S)) \\
		 & \cong \sCart(X, i_* \coprod_S \Delta^0) \\
		 & \cong \sSet(i^* X, \coprod_S \Delta^0) \\
		& \cong \Set( \pi_0 i^* X, S). \qedhere
	\end{align*}
The second statement is easier:
\begin{align*}
	\pi_0 X & \cong \pi_0 \Colim{i_!(\Delta^k \rightarrow X)} \Delta^k \\
	& \cong \Colim{i_!(\Delta^k \rightarrow X)} \pi_0 \Delta^k \\
	& \cong \Colim{i_!(\Delta^k \rightarrow X)} \pi_0 i_!(\Delta^k) \\
	& \cong \pi_0 i_!  \Colim{i_!(\Delta^k \rightarrow X)} \Delta^k \\
	& \cong \pi_0 i_! X.
\end{align*}
The first and last isomorphisms just rewrite $X$ as a colimit over its simplices, the second and fourth isomorphisms follow from the fact that the functors $\pi_0$ and $i_!$ commute with colimits (they are left adjoints), and the third isomorphism is the fact that $\pi_0 \A^{n|0} \cong pt \cong \pi_0 \Delta^k$.
\end{proof}

\begin{example}
	The functor $i$ from Example~\ref{example:sSetasSACS} factors through the category $\sF$ of finite non-empty sets. Thus there is a situation which is entirely analogous to the previous one with simplicial sets replaced with the category $\Pre(\sF)$ of presheaves on $\sF$. This latter is sometimes called the category of {\em symmetric simplicial sets}. 
	
	In fact the category of symmetric simplicial sets should be regarded as a special case of our notion of \sacss; it is the case where the base ring is $\F_1$, the {\em ``field with one element''}. The functors corresponding to $i_!$ and $i^*$ above are then just `base change' and `restriction of scalars' between $\F_1$ and $R$. 
\end{example}

These observations suggest that we should regard \sacss~as an enhanced version of simplicial sets. They are symmetric simplicial sets equipped with additional `face' and `degeneracy' operators which depend on the base ring $R$.



\section{The Picard Category} \label{sec:scommalg-in-sacs}

Recall that we have a series of adjunctions 
\begin{equation*}
	\pi_0 \dashv \const \dashv \Gamma \dashv \codis
\end{equation*}
which relate the topos of sets to the topos of \sacss. The global sections functor $\Gamma$ is a left inverse to the constant presheaf functor, $\Gamma \circ \const \cong id_{\Set}$. Hence we can view the category of sets as consisting of the full subcategory of constant presheaves. For example, the ground ring $R$ induces a ring object $\const(R)$ in $\sCart$, the constant presheaf with value $R$, which we will denote by $R$ to simplify notation. 

Recall that the object $\cO$ is a supercommutative $R$-algebra in $\sCart$ and that $\Gamma(\cO) = \A^{1|1}(\A^0) = R$. 
The $R$-algebra structure may be viewed as coming from the counit map $R = \const \circ \Gamma(\cO) \to \cO$. In this section we develop the internal theory of $\cO$-modules in order to study the invertible $\cO$-modules.


An $\cO$-module will be defined in the usual internal manner: an $\cO$-module is a $\Z/2\Z$-graded superalgebraic cartesian abelian group $M$ with an action by $\cO$. Equivalently, $M$ is a \sacs~such that $M(\A^{n|q})$ is an $\cO(\A^{n|q})$-module for each $\A^{n|q} \in \sA$. Here, since $\cO$ is a {\em super}commutative ring, we mean `module' in the $\Z/2\Z$-graded sense. 

The category $\Mod_\cO$ of $\cO$-modules is a symmetric monoidal abelian category with tensor product $\otimes_{\cO}$ given pointwise:
\begin{equation*}
	(M \otimes_{\cO} N)(\A^{n|q}) := M(\A^{n|q}) \otimes_{\cO(\A^{n|q})} N(\A^{n|q}) 
\end{equation*} 
for $\A^{n|q} \in \sA$. The forgetful functor from $\Mod_\cO$ to $\sCart$ has a left-adjoint  which takes the \sacs~$X$ to the free $\cO$-module $F_\cO(X)$. The value of $F_\cO(X)$ on $\A^{n|q}\in \sA$ is given by $F_{\cO(\A^{n|q})}(X(\A^{n|q}))$, the free $\cO(\A^{n|q})$-module on the set $X(\A^{n|q})$. 

In addition $\Mod_\cO$ has an enrichment in $\sCart$. A map $\A^{n|q} \to \uHom_{\cO}(M, N)$ is defined via
\begin{align*}
	\sCart(\A^{n|q}, \uHom_{\cO}(M, N)) \cong \Hom_{\cO}(F_\cO(\A^{n|q}) \otimes_{\cO}M, N).
\end{align*}
This makes $\Mod_\cO$ into a category enriched in $\sCart$. In fact this enrichment extends to one in the symmetric monoidal category of $\cO$-modules; $\Mod_\cO$ is a closed symmetric monoidal category. To distinguish between the ordinary category of $\cO$-modules and the $\cO$-linear category (enriched in $\cO$-modules) we will denote the former by $\Mod_\cO$ and the latter by $\MMod_\cO$.

Let $\Mod_R$ denote the ordinary category of $R$-modules (in sets). This is a closed symmetric monoidal category and thus an $R$-linear category. Since $\Gamma(\cO) = R$, we obtain an adjunction:
\begin{equation*}
	\cO \otimes_R (-): \Mod_R \rightleftarrows \Mod_\cO: \Gamma, 
\end{equation*}
where the right-adjoint simply applies $\Gamma$ to both the module and ring structure (it is evaluation at $\A^0 \in \sA$). The left-adjoint is given by first viewing a set theoretical $R$-module as a constant (discrete) \sacs~and then tensoring up to obtain an $\cO$-module. As expected, this is a monoidal adjunction with respect to the two symmetric monoidal structures $\otimes_R$ and $\otimes_\cO$, and moreover $\Gamma \circ (\cO \otimes_R (-)) \cong id$ is the identity functor.  

We can do slightly better. Since the above adjunction is monoidal, the functor $\cO \otimes_R(-)$ may be used to enhance the enrichment of $\Mod_R$ in itself into an enrichment in $\Mod_\cO$. Thus for ordinary $R$-modules $M$ and $N$, there exists an $\cO$-module (hence a \sacs) of homomorphisms between them, given by:
\begin{equation*}
	\cO \otimes_R \Hom_R(M,N). 
\end{equation*}
We will denote this new $\Mod_\cO$-enriched category as $\MMod_R$. It has the same objects as $\Mod_R$. The above adjunction now gives rise to a $\Mod_\cO$-enriched  functor:
\begin{equation*}
	\cO \otimes_R (-): \MMod_R \to \MMod_\cO, 
\end{equation*}
which sends an $R$-module $M$ to the $\cO$-module $\cO \otimes_R \const(M)$. Note that the functor $\Gamma$ will not automatically be an enriched functor. 

\begin{lemma}\label{lem:fully-faithful-enriched-inclusion}
	Let $\MMod_R^\textrm{f.g. proj}$ denote the full subcategory of the $\Mod_\cO$-enriched category $\MMod_R$ consisting of those $R$-modules which are finitely generated and projective. Then the restricted $\Mod_\cO$-enriched functor
	\begin{equation*}
		\cO \otimes_R (-):\MMod_R^\textrm{f.g. proj} \to \MMod_\cO
	\end{equation*}
	is fully-faithful (in the enriched sense). 
 \end{lemma}

\begin{proof}
	We must show that the canonical map of $\cO$-modules
\begin{equation*}
	\cO \otimes_R \Hom_R(M,N) \to \uHom_\cO( \cO \otimes_R M, \cO \otimes_R N)
\end{equation*} 	
is an isomorphism if $M$ and $N$ are finitely generated and projective. Note that this is certainly the case if both $M$ and $N$ are finitely generated free $R$-modules. The modules $M= M_0$ and $N=N_0$ are finitely generated and projective if and only if there exist $R$-modules $M_1$ and $N_1$ such that both $M_0 \oplus M_1$ and $N_0 \oplus N_1$ are finitely generated free $R$-modules. Thus the sum of the canonical maps (which is the canonical map of the sums):
\begin{equation*}
	\bigoplus_{i,j=0,1} \cO \otimes_R \Hom_R(M_i,N_j) \to \bigoplus_{i,j=0,1} \uHom_\cO( \cO \otimes_R M_i, \cO \otimes_R N_j)
\end{equation*}
is an isomorphism. The lemma now follows from the observation that in an abelian category a finite collection of maps is a collection of isomorphisms if and only if the direct sum of the collection is an isomorphism.
\end{proof}

Let $\Pic_\cO$ be the {\em Picard category} of $\cO$. It is the full subcategory of $\Mod_\cO$ consisting of the {\em invertible} $\cO$-modules, those $\cO$-modules $M$ such that there exists an $\cO$-module $M'$ with the property that $M \otimes_\cO M' \cong M' \otimes_\cO M \cong \cO$. Let $\PPic_\cO$ denote the corresponding $\Mod_\cO$-enriched subcategory. Similarly $\Pic_R$ will denote the category of invertible $R$-modules and $\PPic_R$ the corresponding $\Mod_\cO$-enriched category. Since $\cO \otimes_R (-)$ is a monoidal functor, it sends invertible objects to invertible objects. Hence we have an induced $\Mod_\cO$-enriched functor:
\begin{equation*}
	\cO \otimes_R (-): \PPic_R \to \PPic_\cO.
\end{equation*}
The following theorem is the main result of this section.

\begin{thm}\label{thm:Picard-equivalence}
	The functor $\cO \otimes_R (-): \PPic_R \to \PPic_\cO$ induces an equivalence of $\Mod_\cO$-enriched symmetric monoidal categories. 
\end{thm}

\noindent We will prove this theorem after a few lemmas. 

\begin{lemma}\label{lem:Objects-of-pic-fgproj}
	The objects of $\Pic_R$ and $\Pic_\cO$ are finitely generated and projective. 
\end{lemma}

\begin{proof}
	This lemma is classical. We will only need the finite generation in the case of $R$-modules. 
	 A categorical proof of this notes that for an object $M \in \Pic_R$ the functor $M\otimes_R (-)$ is an equivalence of categories. Any equivalence of categories preserves the finitely generated projective objects (these are characterized by categorical properties) and moreover the trivial $R$-module $R$ is a finitely generated projective module. Hence the image of $R$ under $M \otimes_R (-)$, that is to say the module $M$, is also a finitely generated projective module. 
\end{proof}

\begin{cor}\label{cor:pic_fullyfaithful}
	The functor $\cO \otimes_R (-): \PPic_R \to \PPic_\cO$ is fully-faithful. \qed
\end{cor}

\begin{lemma}\label{lem:gen-ident-dection}
	Let $f: \cO \to \cO$ be any $\cO$-module map. Assume that $\Gamma(f): \Gamma(\cO)  \to \Gamma(\cO)$ is the identity map, then $f$ is the identity map. 
\end{lemma}

\begin{proof}
	For each $S \in \sA$ we have a commuting diagram 
	\begin{center}
	\begin{tikzpicture}[thick]
			\node (LT) at (0, 1.5) {$\Gamma(\cO) = R$};
			\node (LB) at (0, 0) {$\cO (S)$};
			\node (RT) at (4, 1.5) {$\Gamma(\cO) = R$};
			\node (RB) at (4, 0) {${\cO} (S)$};
			\draw [->] (LT) -- node [left] {$ $} (LB);
			\draw [->] (LT) -- node [above] {$id$} (RT);
			\draw [->] (RT) -- node [right] {$ $} (RB);
			\draw [->] (LB) -- node [below] {$f_S$} (RB);
	\end{tikzpicture}
	\end{center}
	and thus $f_S$ sends $1 \in \cO(S)$ to $1 \in \cO(S)$. As this is a generator of $\cO(S)$ as an $\cO(S)$-module it follows that $f_S$ is the identity for all $S \in \sA$. 
\end{proof}

\begin{proof}[Proof of Thm.~\ref{thm:Picard-equivalence}]
	The functor $\cO \otimes_R (-): \PPic_R \to \PPic_\cO$ is monoidal and, by Corollary~\ref{cor:pic_fullyfaithful}, fully-faithful. It remains to show essential surjectivity, i.e. that every invertible module $M \in \Pic_\cO$ is of the form $\cO \otimes_R L$ for some invertible $R$-module $L$. First note that it is sufficient to prove this under the assumption that $\Gamma(M) \cong R$ is the trivial invertible $R$-module. If $M$ is a general $\cO$-module we may instead consider $N = M \otimes_\cO ( \cO \otimes_R \Gamma(M^{-1}))$, if $N$ is in the essential image then so is $M$. 
	
	Thus without loss of generality assume we have chosen an isomorphism of $R$-modules $\Gamma(M) \cong R$. Let $M'$ be an inverse to $M$. We may also choose an isomorphism $\Gamma(M') \cong R$. Next we will make a few observations. First, from the adjunction $\cO \otimes_R (-) \dashv \Gamma$, we have canonical $\cO$-module homomorphisms
	\begin{align*}
		\cO \cong \cO \otimes_R \Gamma(M) &\to M \\
		\cO \cong \cO \otimes_R \Gamma(M') &\to M'.
	\end{align*}
	Applying $\Gamma$ to either of these yields the identity map of $R$. 
	Next observe that we have a canonical map of $R$-modules in $\sCart$, $R = \Gamma(M) \to M$ and $R = \Gamma(M') \to M'$, where the targets of these are viewed as $R$-modules via the inclusion $R \to \cO$. Again these reduce to the identity maps after applying $\Gamma$. 
	
	Third, we	may tensor together the maps
\[
R \lra{} M' \text{ and } M \lra{\text{id}} M, \text{ over } R \lra{} \cO
\]
to get a map of $\cO$-modules
\[
M \otimes_{R} R \lra{} M \otimes_{\cO} M' \cong \cO.
\]    
Precomposing with the map $\cO \lra{} M$ gives a map of $\cO$-modules
\[
\cO \lra{} M \lra{} \cO.
\]
Since this map reduces to the identity map after applying $\Gamma$, Lemma~\ref{lem:gen-ident-dection} implies that this is the identity map of $\cO$-modules. 

In particular the map $\cO \to M$ is a monomorphism (injective when evaluated on each $S \in \sA$). Tensoring with $M'$ gives a new map
\begin{equation*}
	M' \to M \otimes_\cO M' \cong \cO
\end{equation*} 
which remains a monomorphism since $M'$ is projective (and hence flat). Again this map reduces to the identity after applying $\Gamma$. By symmetry, there exists a monomorphism $M \to \cO$ of $\cO$-modules with the same property (it reduces to the identity after applying $\Gamma$). 

Finally we observe that since the $\cO$-module map $M \to \cO$ is the identity after applying $\Gamma$, it follows that for each $S \in \sA$ the component $M(S) \to \cO(S)$ contains $1 \in \cO(S)$ in its image. Since $1$ is a generator of $\cO(S)$ as an $\cO(S)$-module, it follows that $M(S) \to \cO(S)$ is a surjective map of $\cO(S)$-modules. Consequently the map $M \to \cO$ is both an monomorphism and an epimorphism, hence an isomorphism of $\cO$-modules. In particular $M \cong \cO \otimes_R R$ is in the image of $\PPic_R$.  
\end{proof}

Since sets may be regarded as \sacss~(via the functor $\const$), we may try to regard the $\Mod_\cO$-enriched category $\MMod_\cO$ as a category internal to \sacss. However $\Mod_\cO$ has a large set (or class) of objects, and so is not technically a \sacs. This problem can be avoided for $\PPic_\cO$ since it is essentially small. We will tacitly assume that we have chosen a small set of representative invertible $R$-modules to serve as the set of objects of $\PPic_\cO$. In particular we will regard $\PPic_\cO$ as a symmetric monoidal category internal to \sacss.

\section{Tiny Objects And Internal Homs} \label{sec:tiny}

In this section we explore some of the properties of the internal hom functor from Example~\ref{ex:innerhom}.

\begin{lemma} \label{lma:affine}
There is a natural isomorphism
\[
\underline{\sCart}(\A^{0|1},\A^{n|q}) \cong \A^{n+q|n+q}.
\]
\end{lemma}
\begin{proof}
Because $\A^{n|q} \cong (\A^1)^{\times n}\times(\A^{0|1})^{q}$, we need only check this on $\A^1$ and $\A^{0|1}$. There are two functors $(-)_{\text{ev}},(-)_{\text{odd}}:\sAlg \lra{} \Set$ given by taking the homogeneous parts. These functors are representable and in fact represented by $\Sect(\A^1)$ and $\Sect(\A^{0|1})$ respectively. 
Now we have
\begin{align*}
\sCart(\A^{m|p},\underline{\sCart}(\A^{0|1},\A^{1})) &\cong \sCart(\A^{m|p}\times \A^{0|1}, \A^1) \\ &\cong \Sect(\A^{m|p}\times \A^{0|1})_{\text{ev}} \\ &\cong \Sect(\A^{m|p}) = \A^{1|1}(\A^{m|p})
\end{align*}
and
\begin{align*}
\sCart(\A^{m|p},\underline{\sCart}(\A^{0|1},\A^{0|1})) &\cong \sCart(\A^{m|p}\times \A^{0|1}, \A^{0|1}) \\ &\cong \Sect(\A^{m|p}\times \A^{0|1})_{\text{odd}} \\& \cong \Sect(\A^{m|p}) = \A^{1|1}(\A^{m|p}). \qedhere
\end{align*}
\end{proof}

\begin{definition}
	Let $D$ be a category. An object $x \in D$ is called {\em compact} if $D(x,-): D \to \Set$ commutes with filtered colimits and called {\em tiny} if $D(x,-): D \to \Set$ commutes with \emph{all} small colimits.  Let $D$ be a cartesian closed category. We call an object $x \in D$ {\em cartesian tiny} if $\underline{D}(x,-): D \to D$ commutes with all small colimits, where $\underline{D}(x,-)$ is the internal hom functor. 
\end{definition}

In a presheaf category the tiny objects are precisely those presheaves which are retracts of representables \cite[Prop.2]{MR850528}. 

\begin{prop} \label{tiny}
	If $C$ is a (small) category with finite products, then every tiny object of $\Pre(C)$ is a cartesian tiny object. 
\end{prop}
\begin{proof}
We will show this for representable presheaves. 	
Let $a$ and $b$ be objects of $C$ (also viewed as representable objects of $\Pre(C)$) and let $I \lra{} \Pre(C)$ be a small diagram in $\Pre(C)$ mapping $i \in I$ to $x_i$. 

Colimits in $\Pre(C)$ are computed objectwise, and so we can show that the internal hom out of $a$ commutes with arbitrary (small) colimits by evaluation on an object $b$:
\begin{align*}
\underline{\Pre(C)}(a,\Colim{I} \text{ } x_i)(b) &\cong \Pre(C)(a\times b,\Colim{I} \text{ } x_i) \\
&\cong \Colim{I} \text{ } x_i(a\times b) \\
&\cong \Colim{I} \Pre(C)(a \times b,x_i) \\
&\cong \Colim{I} \text{ } \underline{\Pre(C)}(a, x_i)(b). 
\end{align*}
The second isomorphism uses the fact that the object $a \times b \in C$ is representable. 
\end{proof}

\begin{cor-def}\label{cor:right-adjoint-to-innerhom}
	The internal hom $\usCart(\A^{0|1}, -)$ functor admits a further right adjoint, which we denote $\Omega_{(-)}$. \qed
\end{cor-def}

For any \sacs~$Y$, the \sacs~$\Omega_Y$ has an elementary description:
\begin{equation*}
	 \sCart(\A^{n|q}, \Omega_Y) \cong \sCart( \usCart(\A^{0|1}, \A^{n|q}), Y) \cong Y( \A^{n+q| n+q}).
\end{equation*}
In general we will denote the mapping set $\sCart(X, \Omega_Y) =: \Omega(X ; Y)$. The case $Y = \cO$ is especially important for this paper and in this case we will drop the $\cO$ from our notation; $\Omega := \Omega_\cO$ (see also Example~\ref{ex:Omega}). Another example that will be important later is $\Omega(X ; M)$ for $M$ an invertible $\cO$-module. Recall that $M$ may be viewed as an invertible $R$-module due to Theorem~\ref{thm:Picard-equivalence}. In this case we have the isomorphism
\[
\Omega(X ; M) \cong \Omega(X;\cO) \otimes_R M.
\]
To prove this it suffices to check it on representables for which it is clear.

\begin{cor} \label{cor:functions_on_superpoints}
For any \sacs~$X$ there is an isomorphism of supercommutative rings
\[
\Sect(\underline{\sCart}(\A^{0|1}, X)) \cong \sCart(X, \Omega). \qed
\]
\end{cor}
	



\section{The Action Of The Endomorphisms Of The Super Point} \label{sec:endos}

The \sacs~$\uEnd(\A^{0|1})$ is an internal monoid and consequently $\cO(\uEnd(\A^{0|1}))$ is a coalgebra. We begin this section by describing this coalgebra explicitly with generators and relations. After this we describe it qualitatively by showing that a coaction by this coalgebra on a supercommutative ring is the same information as a connective super cdga structure on the ring. This is a supercommutative algebra which is equipped with an additional grading by the natural numbers together with an odd, degree-one differential. These results are, to a large extent, well-known, and have appeared in a variety of guises throughout the literature. A very general and closely related version appears in the context of super Fermat theories \cite{Carchedi:2012aa}.  Our treatment is heavily influenced by \cite{Stolz-lecture_notes} and \cite{MR2763085}.

The \sacs~$\uEnd(\A^{0|1})$ of endomorphisms of the superpoint (a monoid object) naturally acts on the internal mapping \sacs~$\usCart(\A^{0|1}, X)$ for any \sacs~$X$. In the second part of this section we provide conditions on $X$ under which this action leads to a coaction by $\cO(\uEnd(\A^{0|1}))$. These results do not seem to have appeared in the literature.

The most direct approach to the monoid structure on $\uEnd(\A^{0|1})$ is via the $S$-point formalism. Here $S \in \sA$ is some unspecified representable \sacs. The $S$-points of $\uEnd(\A^{0|1}) = \usCart(\A^{0|1}, \A^{0|1})$ are, by construction, the maps $S \times \A^{0|1} \to \A^{0|1}$. This, in turn, is equivalent to a map $S \times \A^{0|1} \to S \times \A^{0|1}$, which commutes with the projection to $S$. This latter description is convenient, as the monoid structure on $\uEnd(\A^{0|1})$ is given by composition. Since $S$ and $S \times \A^{0|1}$ are representable, we may equivalently describe such data by passing to the rings of global functions. We see that an $S$-point of $\uEnd(\A^{0|1})$ is given by an $\cO(S)$-algebra map
\begin{align*}
	\cO(S)[\varepsilon] &\to \cO(S)[\varepsilon] \\
	\varepsilon & \mapsto s_{\text{ev}} \cdot \varepsilon + s_{\text{odd}}.
\end{align*}
where $s_{\text{ev}}$ and $s_{\text{odd}}$ are even, respectively odd, elements of $\cO(S)$. 
This description makes explicit the identification $\underline{\sCart}(\A^{0|1}, \A^{0|1}) \cong \A^{1|1}$ from Lemma \ref{lma:affine}. Moreover we have 
\begin{align*}
	&\cO(\underline{\sCart}(\A^{0|1}, \A^{0|1}) \times \underline{\sCart}(\A^{0|1}, \A^{0|1}))  \\ & \cong \cO(\underline{\sCart}(\A^{0|1}, \A^{0|1})) \otimes_R \cO(\underline{\sCart}(\A^{0|1}, \A^{0|1}))
\end{align*}
and hence the monoid structure of $\underline{\sCart}(\A^{0|1}, \A^{0|1})$ induces a comonoid structure for $\cO(\underline{\sCart}(\A^{0|1}, \A^{0|1}))$.

It follows immediately from the formula for composition of affine transformations in one variable that the global functions of the multiplication map for the monoid $\uEnd(\A^{0|1})$ are given by the map
\begin{align*}
	R[x,\e] & \lra{m^*} R[x_1,x_2,\e_1,\e_2] \\
	x & \mapsto x_1x_2 \\
	\e & \mapsto \e_1+x_1\e_2.
\end{align*}

This implies:
\begin{prop}
There is an isomorphism of monoidal \sacss
\[
\uEnd(\A^{0|1}) 
\cong \A^{0|1} \rtimes \A^1,
\]
where $\A^{1}$ acts on $\A^{0|1}$ by scalar multiplication.
\end{prop}

\begin{definition}\label{def:supercdga}
	A {\em supercommutative differential graded algebra} (super cdga) is a supercommutative algebra $A$ equipped with
	\begin{itemize}
		\item a {\em grading}, i.e. a collection of $R$-module direct summands $A_n \subseteq A$ for each $n \in \Z$ such that $A_p \cdot A_q \subseteq A_{p+q}$, and as $R$-modules $A \cong \bigoplus_n A_n$; and
		\item a {\em differential}, i.e. an odd derivation $d: A \to A$, which squares to zero, $d^2 = 0$. 
	\end{itemize}
	We require that the derivation has {\em degree one}, which means $d(A_p) \subseteq A_{p+1}$. A super cdga is {\em connective} if $A_p = 0$ for $p < 0$. 
	
	A {\em weakly graded super cdga} $B$ is defined identically to a super cdga except that we require $B \cong \prod_n B_n$, the direct product of isotypical factors, rather than the direct sum. 
\end{definition}

\begin{prop} \label{cdgastructure}
Let $A$ be a supercommutative algebra (such as $\Sect(Y)$). A coaction by $\Sect(\A^{0|1} \rtimes \A^1)$ is equivalent to a connective super cdga structure on $A$. 
\end{prop}
\begin{proof}
The semidirect product $\A^{0|1} \rtimes \A^1$ admits a canonical section
\[
\A^{0|1} \lra{} \A^{0|1} \rtimes \A^1 \longleftarrow \A^1,
\]
which allows us to break the action into its constituent parts.

A coaction of $\Sect(\A^1) \cong R[x]$ is equivalent to a connective grading by $\mathbb{N}$, with the elements $a \in A$ of degree $k$ being those where the coaction map sends $a \mapsto a\otimes x^k$. Here by a grading we mean, as above, that $A \cong \oplus A_n$, the direct sum of factors, and not the direct product.  

A coaction of $\Sect(\A^{0|1})$ is equivalent to an odd differential. Again, for $a \in A$ the coaction map sends $a \mapsto a + a_1\e$. We set $da = a_1$. The fact that $d$ is a differential follows from the associativity of the action.

Now we combine these actions in a twisted way in $\A^{0|1} \rtimes \A^1$. We can check that this tells us that the differential increases degree. The associativity diagram for the coaction is the following:
\begin{center}
\begin{tikzpicture}
	\node (LT) at (0, 1.5) {$A$};
	\node (LB) at (0, 0) {$A \otimes \Sect(\A^{0|1} \rtimes \A^1)$};
	\node (RT) at (6, 1.5) {$A\otimes \Sect(\A^{0|1} \rtimes \A^1)$};
	\node (RB) at (6, 0) {$A\otimes \Sect(\A^{0|1} \rtimes \A^1) \otimes \Sect(\A^{0|1} \rtimes \A^1)$.};
	\draw [->] (LT) -- node [left] {$\mu^*$} (LB);
	\draw [->] (LT) -- node [above] {$\mu^*$} (RT);
	\draw [->] (RT) -- node [right] {$1 \otimes m^*$} (RB);
	\draw [->] (LB) -- node [below] {$\mu^* \otimes 1$} (RB);
\end{tikzpicture}
\end{center}
Now if $a \in A$ is degree $k$, we have
\begin{center}
\begin{tikzpicture}
	\node (LT) at (0, 1.5) {$a$};
	\node (RT) at (4.5, 1.5) {$ax^k+(da)x^k\e$};
	\node (RB) at (4.5, 0) {$ax_{1}^{k}x_{2}^{k}+(da)x_{1}^{k}x_{2}^{k}\e_1 - (da)x_{1}^{k+1}x_{2}^{k}\e_2$};
	\draw [|->] (LT) -- node [above] {$ $} (RT);
	\draw [|->] (RT) -- node [right] {$ $} (RB);
\end{tikzpicture},
\end{center}
which is completely determined by the formula for $m^*$. Going around the other way we discover that there must be an equality $\mu^*(da) = (da)x^{k+1}$. That is, the differential increases degree by $1$.
\end{proof}


\begin{cor}
There is an equivalence of categories between the category of supercommutative algebras with coactions by $\Sect(\A^{0|1} \rtimes \A^1)$ and the category of super cdga's. \qed
\end{cor}

Now we calculate the effect of $\cO(-)$ on the action map
\[
\mu:\uEnd(\A^{0|1}) \times \usCart(\A^{0|1},\A^{n|q}) \lra{} \usCart(\A^{0|1},\A^{n|q}).
\]
Note that, since everything involved is affine, this gives a coaction of $\cO(\uEnd(\A^{0|1}))$ on $\cO(\usCart(\A^{0|1},\A^{n|q}))$. We write
\[
\cO(\usCart(\A^{0|1},\A^{n|q})) \cong R[x_1,\ldots,x_n,\be_1,\ldots,\be_q,\bx_1,\ldots,\bx_n,\e_1,\ldots,\e_q],
\]
where $x_i$ and $\be_i$ are even and $\bx_i$ and $\e_i$ are odd. We use this notation because $\bx_i$ is induced by $x_i \in \cO(\A^{n|q})$ and $\be_i$ is induced by $\e_i \in \cO(\A^{n|q})$.

\begin{prop} \label{prop:calc}
The coaction of $\cO(\uEnd(\A^{0|1})) \cong R[x,\e]$ on $\cO(\usCart(\A^{0|1},\A^{n|q}))$ maps
\begin{align*}
x_i &\mapsto x_i+\bx_i\e \\
\be_i &\mapsto \be_ix \\
\bx_i &\mapsto \bx_ix \\
\e_i &\mapsto \e_i+\be_i\e.
\end{align*}
\end{prop}
\begin{proof}
Since $\A^{n|q} \cong (\A^1)^n\times (\A^{0|1})^q$ it suffices to check this on $\A^1$ and $\A^{0|1}$. We prove this for $\A^1$, the case of $\A^{0|1}$ having already been treated during our explicit description of the coalgebra structure of $\cO(\uEnd(\A^{0|1}))$. Let $T$ be a $\Z/2$-graded commutative $R$-algebra. Using the functor of points a map
\[
R[x,\e] \lra{} T
\]
mapping $x \mapsto t_x$ and $\e \mapsto t_{\e}$ corresponds to the map
\[
T[\e] \lra{} T[\e]: \e \mapsto t_x\e + t_{\e}.
\] 
Now $\cO(\usCart(\A^{0|1},\A^1)) \cong R[x_1,\bx_1]$. A map
\[
R[x_1, \bx_1] \lra{} T
\]
mapping $x_1 \mapsto t_{x_1}$ and $\bx_1 \mapsto t_{\bx_1}$ corresponds to the map
\[
T[x] \lra{} T[\e]: x \mapsto t_{x_1} + t_{\bx_1} \e.
\]
Composing these gives the map
\[
T[x] \lra{} T[\e]: x \mapsto (t_{x_1}+t_{\bx_1}t_{\e}) + t_{\bx_1}t_x \e.
\]
Thus the coaction is the map
\[
R[x_1,\bx_1] \lra{} R[x,\e]\otimes_R R[x_1,\bx_1]
\]
mapping $x_1 \mapsto x_1+\bx_1\e$ and $\bx_1 \mapsto \bx_1x$.
\end{proof}
The super cdga structure on $\cO(\usCart(\A^{0|1},\A^{n|q}))$ thus has $x_i$ even of degree $0$, $\bx_i$ odd of degree $1$, $\e_i$ odd of degree $0$, and $\be_i$ even of degree $1$.

Now let $Y$ be a \sacs~with an action of $\uEnd(\A^{0|1})$:
\[
\uEnd(\A^{0|1}) \times Y \lra{\mu} Y.
\]
Applying global functions gives a map
\begin{equation*}
	\mu^*: \cO(Y) \to \cO(\uEnd(\A^{0|1})  \times Y).
\end{equation*}
When $Y$ is representable, the codomain decomposes as a tensor product 
\[
\cO(\uEnd(\A^{0|1})  \times Y) \cong \cO(\uEnd(\A^{0|1})) \otimes_R \cO(Y),
\]
and hence in this case $\Sect(Y)$ becomes a supercomodule for the supercoalgebra $R[x,\e]$ described above. In general taking global functions fails to turn products into tensor products, and so we would not generally expect such a coaction. We begin by analyzing the general case in order to see just how bad things can get. Then we show that for a \sacs~$X$ that is a finite colimit of representables there is a genuine coaction on 
\[
\cO(\usCart(\A^{0|1},X)).
\]
Also we show that in the case when there exists $N \in \N$ such that $X$ is a colimit of representables of the form $\A^{n|0}$ where $n < N$ there is an induced coaction on $\cO(\usCart(\A^{0|1},X))$.


Let $Y$ be a \sacs~with an action by $\uEnd(\A^{0|1})$. Then $Y = \Colim{I} \text{ } \A^{n|q}$ and we have the sequence of isomorphisms
\begin{align*}
\cO(\uEnd(\A^{0|1}) \times Y) &\cong \cO(\uEnd(\A^{0|1}) \times \Colim{I} \text{ } \A^{n|q}) \\
&\cong \cO(\Colim{I} \text{ }(\uEnd(\A^{0|1}) \times  \A^{n|q})) \\
&\cong \Lim{I} \text{ }\cO((\uEnd(\A^{0|1}) \times  \A^{n|q})).
\end{align*}

Since the action map $\mu$ is natural in $Y$ and $\Sect(\A^{n|q})$ admits an coaction map, 
this implies that the map of rings
\[
\cO(Y) \lra{\mu^*} \cO(\uEnd(\A^{0|1}) \times Y)
\]
is a limit of coaction maps
\[
\Lim{I} \text{ }\cO(\A^{n|q}) \lra{} \Lim{I} \text{ }\cO((\uEnd(\A^{0|1}) \times  \A^{n|q})).
\]
Thus Proposition \ref{prop:calc} provides a formula for this map. An element of the ring $\cO(Y)$ is a compatible family of polynomials in $\cO(\A^{n|q})$ as $n$ and $q$ vary. Let $\mathbf{x} = \{x_1, \ldots, x_n\}$ and $\boldsymbol{\e} = \{\e_1,\ldots,\e_q\}$. If $(f_i(\mathbf{x},\boldsymbol{\e}))_{i \in I}$ is a compatible family then 
\[
\mu^*((f_i(\mathbf{x},\boldsymbol{\e}))_{i \in I}) = (\mu^*f_i(\mathbf{x},\boldsymbol{\e}))_{i \in I}, 
\]
where the $\mu^*$ on the right is the coaction of $\cO(\uEnd(\A^{0|1}))$ on $\cO(\A^{n|q})$.

Now we explain two cases in which the limit of coaction maps induces a coaction by $\cO(\uEnd(\A^{0|1}))$. 

\begin{prop}
Let $X = \Colim{I} \text{ } \A^{n|q}$ where $I$ is a finite category. Then there is an isomorphism
\[
\cO(\uEnd(\A^{0|1}))\otimes_R \cO(\usCart(\A^{0|1},X)) \cong \cO(\uEnd(\A^{0|1}) \times \usCart(\A^{0|1},X)).
\]
This implies that the ring of functions applied to the action map gives a coaction for ``finite" \sacss.
\end{prop}
\begin{proof}
The key point here is that $\cO(\uEnd(\A^{0|1}))$ is flat as an $R$-module. The underlying module is an (infinitely generated) free $R$-module.

Now we have the sequence of isomorphisms
\begin{align*}
\cO(\uEnd(\A^{0|1}) \times &\usCart(\A^{0|1},X)) \\ &\cong \cO(\uEnd(\A^{0|1}) \times \usCart(\A^{0|1},\Colim{I}\text{ }\A^{n|0})) \\ 
&\cong \cO(\uEnd(\A^{0|1}) \times \Colim{I}\text{ }\usCart(\A^{0|1},\A^{n|0})) \\
&\cong \cO(\Colim{I}\text{ }(\uEnd(\A^{0|1}) \times \usCart(\A^{0|1},\A^{n|0}))) \\
&\cong \Lim{I}\text{ }\cO(\uEnd(\A^{0|1}) \times \usCart(\A^{0|1},\A^{n|0})) \\
&\cong \Lim{I}\text{ }\cO(\uEnd(\A^{0|1})) \otimes_R \cO(\usCart(\A^{0|1},\A^{n|0})) \\
&\cong \cO(\uEnd(\A^{0|1})) \otimes_R \Lim{I}\text{ }\cO(\usCart(\A^{0|1},\A^{n|0})) \\
&\cong \cO(\uEnd(\A^{0|1})) \otimes_R \cO(\usCart(\A^{0|1},X)).
\end{align*}
The second isomorphism follows from the fact that $\A^{0|1}$ is cartesian tiny. The third is because colimits distribute with products in a topos. The fifth is because the objects are affine and the sixth uses the fact that $\cO(\uEnd(\A^{0|1}))$ is flat.
\end{proof}

\begin{prop} \label{prop:finitedim}
Assume that there exists $N \in \N$ such that $X = \Colim{I} \text{ } \A^{n|0}$ with $n < N$. Then the ring of functions on the action map factors through the tensor product. Thus the action map of $\uEnd(\A^{0|1})$ on $\usCart(\A^{0|1},X)$ induces a coaction on global functions.
\end{prop}
\begin{proof}
The functor $\cO(-)$ applied to the action map gives
\[
\Lim{I}\text{ }\cO(\usCart(\A^{0|1},\A^{n|0}) \lra{} \Lim{I}\text{ }\cO(\End(\A^{0|1}) \times \usCart(\A^{0|1},\A^{n|0})),
\]
which is an inverse limit of coactions. We see from Proposition \ref{prop:calc} that the action of $\End(\A^{0|1})$ on $\usCart(\A^{0|1},\A^{n|0}))$ induces a grading on $\cO(\usCart(\A^{0|1},\A^{n|0}))) \cong \cO(\A^{n|n})$. The maximal element in the grading has degree $n$. We claim that this implies that there is a factorization of the map above through
\[
\Lim{I}\text{ }\cO(\usCart(\A^{0|1},\A^{n|0}) \lra{} \cO(\End(\A^{0|1}))
	\otimes_R \left(\Lim{I}\text{ }\cO(\usCart(\A^{0|1},\A^{n|0}))\right).
\]
Since $\cO(\uEnd(\A^{0|1}))$ is a polynomial ring (and not a power series ring), 
a factorization exists as long as there is no element in $\Lim{I}\text{ }\usCart(\A^{0|1},\A^{n|0}))$ that has unbounded degree. 
\end{proof}

\begin{example}
Consider the \sacs $\Coprod{n \geq 0} \A^{n|0}$, which is the disjoint union of (non-super) affine spaces, one of each dimension. There is an isomorphism
\[
\cO(\usCart(\A^{0|1},\Coprod{n \geq 0} \A^{n|0})) \cong \Prod{n}\cO(\A^{n|n})
\]
and so this ring has unbounded degree. In particular it contains the element $(1,\varepsilon_1,\varepsilon_1\varepsilon_2,\varepsilon_1\varepsilon_2\varepsilon_3, \ldots)$. Thus no factorization as above could exist for this ring and so the action of $\uEnd(\A^{0|1})$ on $\usCart(\A^{0|1},\Coprod{n \geq 0} \A^{n|0})$ does not induce a coaction after taking the ring of functions. 
\end{example}

For a \sacs~$X$ satisfying the hypotheses of either of the propositions above, let $\Omega^*(X)$ be the super cdga associated to $\Sect(\underline{\sCart}(\A^{0|1},X))$ with the coaction by $\Sect(\A^{0|1} \rtimes \A^1)$. So if $u$ is the forgetful functor from super cdga's to $\sAlg$ we have $u \Omega^*(X) = \Sect(\underline{\sCart}(\A^{0|1},X)) = \Omega(X)$.

\section{Polynomial Forms via Superalgebraic Cartesian Sets} \label{sec:forms}

In \cite{Sull}, Sullivan introduced a simplicial commutative differential graded algebra (cdga) called $\Omega^{*}_{\bullet}$. It is defined on $n$-simplices by the formula
\[
\sSet(\Delta_n,\Omega^{*}_{\bullet}) \cong R[x_1, \ldots, x_n, dx_1, \ldots dx_n],
\]
where $|x_i| = 0$. This is the cdga of K\"ahler differential forms on the polynomial algebra $R[x_1, \ldots, x_n,]$. 
The simplicial maps are built just as in the functor $i$ introduced in Subsection~\ref{sec:SACS_simplicial}. 

The $n$-simplices of the simplicial cdga in fact have the structures of a super cdga (a cdga with a $\Z/2$ grading and an odd degree $1$ differential). The elements $x_i$ are even of degree $0$ and the elements $dx_i$ are odd of degree $1$. 

For any simplicial set $X$ the set of maps $\sSet(X, \Omega^{*}_{\bullet}) =: \Omega_R^*(X)$ is a commutative differential graded $R$-algebra, which is only weakly graded if $X$ is infinite dimensional (c.f. Def.~\ref{def:supercdga}). This cdga has a concrete description. An element consists of a compatible choice $\{\omega_\sigma\}_{\sigma \in X}$ of polynomial K\"ahler differential forms for each simplex $\sigma$ of $X$. This collection is required to be compatible with restriction maps in the obvious way. 

When $R$ is a $\Q$-algebra, the simplicial cdga $\Omega^{*}_{\bullet}$ has the property that, for a simplical set $X$, 
\[
H^*(\sSet(X, \Omega^{*}_{\bullet})) \cong H^*(X,R),
\] 
where $H^*(X,R)$ is the singular cohomology of $X$ with coefficients in the ring $R$ (c.f. \cite[Thm~7.1]{Sull}). The cdga $\Omega_\Q^*(X)$ is Sullivan's rational polynomial differential forms.

There is a forgetful functor $u$ from the category of super cdga's to $\sAlg$. The simplicial supercommutative algebra
\[
u\Omega^{*}_{\bullet}: \Delta^{op} \lra{} \sAlg.
\]
will play an in important role in this section, where we prove that for any simplicial set $X$ there is a natural isomorphism of supercommutative algebras
\[
\Sect(\underline{\sCart}(\A^{0|1}, i_!X)) \cong u\sSet(X, \Omega^{*}_{\bullet}).
\]


We begin by studying the relationship between $\Omega$ and $\Omega^{*}_{\bullet}$.

\begin{prop}
There is a natural isomorphism of simplicial supercommutative algebras (where $\Omega$ is the \sacs~from Example~\ref{ex:Omega})
\[
u\Omega^{*}_{\bullet} \cong i^* \Omega.
\]
\end{prop}
\begin{proof}
Evaluating on $\Delta^n$ gives
\begin{align*}
\sSet(\Delta^n, i^* \Omega) &\cong \sCart(i_!\Delta^n, \Omega) \\ &\cong \sCart(\A^n, \Omega) \\&\cong \Sect(\underline{\sCart}(\A^{0|1},\A^n)) \\ &\cong \Sect(\A^{n|n}) \\ &\cong R[x_1,\ldots,x_n,\e_1, \ldots, \e_n] \\ &\cong u\Omega^{*}_{\bullet}(\Delta^n).
\end{align*}
\end{proof}

As a special case we get the following corollary:

\begin{cor} \label{maincor}
For any simplicial set $X$ there is an isomophism
\[
\Sect(\underline{\sCart}(\A^{0|1}, i_!X)) \cong u\sSet(X, \Omega^{*}_{\bullet}).
\]
\end{cor}
\begin{proof}
There are isomorphisms
\begin{align*}
u\sSet(X, \Omega^{*}_{\bullet}) &\cong \sSet(X, u\Omega^{*}_{\bullet}) \\
&\cong \sSet(X, i^*\Omega) \\ &\cong \sCart(i_!X,\Omega) \\ &\cong \Sect(\underline{\sCart}(\A^{0|1}, i_!X)).
\end{align*}
The last isomorphism is an application of Corollary \ref{cor:functions_on_superpoints}.
\end{proof}

In other words, for any simplicial set $X$ the supercommutative algebra underlying the commutative differential graded algebra $\Omega^*_R(X)$ of polynomial differential forms over the ring $R$ on $X$ is naturally isomorphic to the ring of functions on the internal mapping object $\usCart(\A^{0|1}, i_!X)$. 

\begin{prop}
For $X$ a  simplicial set, there is an isomorphism of (weakly graded) super cdga's
\[
\Omega^*(i_!X) \cong \sSet(X,\Omega^{*}_{\bullet}).
\]
\end{prop}
\begin{proof}
	We will first consider the case of $X$ a finite dimensional simplicial set.
In Corollary \ref{cor:functions_on_superpoints}, we showed that the above is an isomorphism of the underlying supercommutative algebras. Here we lift this to the category of super cdga's. 

The forgetful functor $u$ creates limits and $i_!$ preserves colimits. This implies that $X = \Colim{\Delta^k \rightarrow X}\Delta^k$ satisfies the conditions of Proposition \ref{prop:finitedim}. Now there are isomorphisms
\begin{align*}
\Omega^*(i_!X) &\cong \Omega^*(\Colim{i_!(\Delta^k \rightarrow X)}i_!\Delta^k) \\ &\cong \Omega^*(\Colim{i_!(\Delta^k \rightarrow X)}\A^k) \\ &\cong \Lim{i_!(\Delta^k \rightarrow X)}\Omega^*(\A^k).
\end{align*}

Thus it suffices to prove the result for $\A^k \cong i_!\Delta^k$. Now this follows from Proposition \ref{prop:calc}. Thus $\bx_i$ corresponds to $dx_i$.

Now let $X$ be an infinite dimensional simplicial set. We may write it as a colimit of its finite dimensional skeleta. This implies that $\Omega(X)$ admits a sequential inverse limit of coactions. We have
\begin{equation*}
	\Omega (X) \cong \Prod{i \in \N} \Omega^i(X).
\end{equation*}
This fails to have a super cdga structure only in that it is the direct product over isotypical factors, rather than the direct sum. Hence it is weakly graded.  
\end{proof}

\section{Geometries On The Superpoint} \label{sec:geometries}

In this section we will study several possible geometries that can be placed on the superpoint. Each of these geometries will give rise to a slightly different breed of supersymmetric $0|1$-dimensional quantum field theory. 

Following the lead of HKST we will define a  geometry in the spirit of Felix Klein's Erlangen program. That is to say a {geometry} is completely specified by its group symmetries, which is a subgroup of $\uAut(\A^{0|1})$. In fact the most natural thing which acts on $\A^{0|1}$ is the {\em monoid} of endomorphisms; we don't see a compelling reason to limit ourselves to sub{\em groups}.  

\begin{definition}\label{def:geometry}
	A {\em geometry} on $\A^{0|1}$ is a submonoid $\M$ of the monoid $\uEnd(\A^{0|1})$ of endomorphisms (in \sacss). 
\end{definition}

\noindent There are five geometries that we will explore below:
\begin{enumerate}
	\item $\M = \uEnd(\A^{0|1}) \cong \A^{0|1} \rtimes \A^1$ is the full endomorphism monoid. We call this geometry {\em pre-topological}. 
	\item $\M = \uAut(\A^{0|1}) \cong \A^{0|1} \rtimes \G_m$ is the maximal subgroup. We call this geometry {\em topological}.
	\item $\M = \A^{0|1} \rtimes \Z/2\Z$. Following HKST we call this geometry {\em Euclidean}.
	\item $\M  = \A^{0|1} \times 1$. We call this geometry {\em oriented Euclidean}. 
	\item $\M = 1$. We call this geometry {\em fully-rigid}. 
\end{enumerate}

The geometries (submonoids) include into each other in the following way:
\[
1 \subset \A^{0|1} \times 1 \subset \A^{0|1} \rtimes \Z/2\Z \subset \A^{0|1} \rtimes \G_m \subset \A^{0|1} \rtimes \A^1,
\]
where we have abused notation and written $\G_m$ for $\Sect^*(\G_m)$. On global functions these inclusions correspond to the maps of supercommutative bialgebras
\[
R[x,\e] \lra{x \mapsto x} R[x,x^{-1},\e] \lra{x \mapsto (1,-1)} (R\times R)[\e] \lra{\pi_1} R[\e] \lra{\e \mapsto 0} R.
\]

\begin{cor} The following are consequences of the proof of Proposition \ref{cdgastructure}:
\begin{enumerate}
\item A supercommutative algebra with a coaction by $\Sect(\A^{0|1} \rtimes \G_m)$ is a $\Z$-graded super cdga. 

\item A supercommutative algebra with a coaction by $\Sect(\A^{0|1} \rtimes \Z/2)$ is a $\Z/2$-graded super cdga. 

\item A supercommutative algebra with a coaction by $\Sect(\A^{0|1})$ is a supercommutative algebra with an odd differential.
\end{enumerate}
\end{cor}

Next we define the notion of $\M$-structure on a \sacs~that is abstractly isomorphic to the superpoint.

\begin{definition}
Let $X$ be a \sacs~that is abstractly isomorphic to the superpoint $\A^{0|1}$. An \emph{$\M$-prestructure} on $X$ is a subfunctor
\[
\Gamma \subseteq \underline{\sCart}(\A^{0|1},X)
\]
with the property that $\Gamma$ is closed under the action of $\M$:
\[
\Gamma \cdot \M = \Gamma.
\]

An $\M$-isometry between two \sacss~equipped with $\M$-prestructures, $(X,\Gamma)$ and $(X',\Gamma ')$, is a map $X \lra{f} X'$ such that $f_*\Gamma \subseteq \Gamma '$. Thus $(X, \Gamma)$ is isomorphic to $(X',\Gamma ')$ if there is an isomorphism $X \lra{f} X'$ such that $f_*\Gamma = \Gamma '$.
\end{definition}

\begin{example}
The superpoint $\A^{0|1}$ has a canonical $\M$-prestructure given by
\[
\M \subseteq \underline{\sCart}(\A^{0|1},\A^{0|1}).
\]
\end{example}

\begin{definition}
An $\M$-prestructure $(X,\Gamma)$ is an \emph{$\M$-structure} if there exists an isomorphism 
\[
(X,\Gamma) \cong (\A^{0|1}, \M).
\]
\end{definition}

There is an action of $\M$ on the \sacs
\[
\underline{\sCart}(\A^{0|1},X)
\]
given by precomposition. We may consider the (categorical) quotient \sacs
\[
\underline{\sCart}(\A^{0|1},X)/\M.
\]
When $\M$ is a group and thus a subgroup of $\A^{0|1}\rtimes \G_m$, we may consider its normalizer in $\A^{0|1}\rtimes \G_m$ defined by the formula
\[
N(\M)(\A^{n|q}) = \{g \in (\A^{0|1}\rtimes \G_m)(\A^{n|q}) | g\M(\A^{n|q})g^{-1} = \M(\A^{n|q}).
\]
In each of the above cases the normalizer is the whole of $\A^{0|1}\rtimes \G_m$.

It is clear that $N(\M)$ acts on
\[
\underline{\sCart}(\A^{0|1},X)/\M.
\]
When $\M$ is normal in $N(\M)$, the action factors through $N(\M)/\M$.

\begin{example}
The most interesting example of this is the Euclidean geometry
\[
\M = \A^{0|1}\rtimes \Z/2.
\]
We have $N(\M) = \A^{0|1}\rtimes \G_m$ and when $2 \in R^{\times}$, $\Z/2 \cong \G_m[2]$, and hence
\[
N(\M)/\M \cong \G_m/\G_m[2] \cong \G_m.
\]
Note that an action of $\G_m$ on $\Spec(R)$ that factors through $\G_m/\G_m[2]$ is equivalent to a grading by the even integers.
\end{example}

\begin{remark}
For each of the other geometries this is quite elementary. When $\M = 1$, there is an action of $\A^{0|1}\rtimes \G_m$. When $\M = \A^{0|1}$, there is an action of $\G_m$. 
\end{remark}

\section{Superalgebraic Cartesian Quantum Field Theories} \label{sec:SQFT2}

We are now in a position to define supersymmetric $0|1$-dimensional quantum field theories over an arbitrary \sacs. For each geometry $\M$ (discussed in the last section) and each \sacs~$X$, we construct $\Bord_{(\M, X)}^{0|1}$, the symmetric monoidal 0-category (internal to \sacss) consisting of  $0|1$-dimensional bordisms equipped with $\M$-structures and maps to $X$. 
As a symmetric monoidal 0-category is just a commutative monoid, $\Bord_{(\M, X)}^{0|1}$ is just a commutative monoid object in \sacss. We will describe it in more detail in just a moment. 

The target of a field theory is another symmetric monoidal category, which in this 0-dimensional case means another commutative monoid (internal to \sacss). The target of a {\em quantum} field theory (as opposed to a classical field theory or other variety of field theory) should have some further mechanism implementing the physical concept of {\em superposition}. This can be accomplished by requiring the target category to have not just a multiplicative (i.e. monoidal) structure, but to also have an additive structure. In the classical context of the Atiyah-Segal axioms this is the direct sum operation on the target category of vector spaces. In the case at hand it means that our target should be a ring (or at least a rig). A natural choice is the ring $\cO = \A^{1|1}$.

	A supersymmetric $0|1$-dimensional $\M$-quantum field theory over a \sacs~$X$ is then defined to be a homomorphism
	\begin{equation*}
		Z: \Bord_{(\M, X)}^{0|1} \to \cO
	\end{equation*}
of commutative monoids in \sacss.

As usual, the easiest way to describe $\Bord_{(\M, X)}^{0|1}$ and the homomorphism $Z$ is via the formalism of $S$-points. For each representable $\A^{n|q} \in \sA$, and each map $f:\A^{n|q} \to \Bord_{(\M, X)}^{0|1}$, the field theory $Z$ should associate a map $Z(f): \A^{n|q} \to \cO$. That is to say we obtain a function $Z(f) \in \cO(\A^{n|q})$.  Maps $\A^{n|q} \to \Bord_{(\M, X)}^{0|1}$ are obtained by considering $\A^{n|q}$-families of bordisms. 

Each of the structures we place on our bordisms, the $\M$-structure, the map to $X$, the supercommutative ring of functions, indeed even the enrichment in \sacss, should be thought of as enhancements we give to an underlying topological 0-bordism. The category $\Bord_{(\M, X)}^{0|1}$ should have a forgetful functor to $\Bord^0$. 
Hence every $0|1$-dimensional bordism is equivalent to a finite disjoint union of  superpoints, and we will say that this bordism is equipped with an $\M$-structure if each component superpoint has an $\M$-structure. Two such bordisms will be equivalent if they are related by $\M$-isometries, where an $\M$-isometry between bordisms is a permutation followed by $\M$-isometries on each factor. This definition ensures there is a  forgetful functor to $\Bord^0$. 

\begin{definition}
	The superalgebraic cartesian commutative monoid $\Bord_{(\M, X)}^{0|1}$ associates to $\A^{n|q} \in \sA$ the set of equivalence classes of $\A^{n|q}$-families of $0|1$-dimensional bordisms equipped with $\M$-structure, with the equivalence relation of $\M$-isometry. 
\end{definition}

Hence an $\A^{n|q}$-family of such bordisms induces a map $\A^{n|q} \to \Bord_{(\M, X)}^{0|1}$. Thus a field theory $Z$ will associate to each such family
\begin{equation*}
	f:\A^{n|q} \times Y^{0|1} \to X
\end{equation*}
a function $Z(f) \in \cO(\A^{n|q})$ (where $Y^{0|1} \cong \coprod_k \A^{0|1}$ for some $k$, and is equipped with an $\M$-structure). If two  $\A^{n|q}$-families of bordisms are related by an $\A^{n|q}$-family of $\M$-isometries, then the associated functions will be the same. 
This gives rise to the following explicit description of $\Bord_{(\M, X)}^{0|1}$:



\begin{prop} \label{prop:Bord_as_free_comm_monoid}
	The \sacs~$\Bord_{(\M, X)}^{0|1}$ is given by the quotient 
	\begin{align*}
		\Bord_{(\M, X)}^{0|1} &\cong  \coprod_{k \in \N} \left( \usCart(\coprod_k\A^{0|1}, X) / \M \wr \Sigma_k \right) \\
		& \cong \coprod_{k \in \N} \left( \prod_k [\usCart(\A^{0|1}, X)/ \M]  / \Sigma_k   \right).
	\end{align*}
\end{prop}

 The commutative monoid structure on $\Bord_{(\M, X)}^{0|1}$ is induced by the disjoint union operation on bordisms, and combining this with the above explicit description we obtain:

\begin{cor}
	$\Bord_{(\M, X)}^{0|1}$ is the free commutative monoid generated by the \sacs
	\begin{equation*}
		\usCart(\A^{0|1}, X)/ \M.
	\end{equation*}
\end{cor}

\begin{cor}
	The supersymmetric $0|1$-dimensional $\M$-quantum field theories over a \sacs~$X$ (with values in $\cO$) are in natural bijection with the set of $\M$-invariant functions on $\usCart(\A^{0|1}, X)$. 
\end{cor}

Note that this implies that the field theories naturally have the structure of a commutative ring. Using the description from this last corollary and our previous calculations we may now identify the supersymmetric $0|1$-dimensional $\M$-quantum field theories over a simplicial set. 

\begin{thm} \label{thm:field-theories}
	Let $X$ be a finite dimensional simplicial set. Then the set of supersymmetric $0|1$-dimensional $\M$-quantum field theories over $i_! X$ is naturally isomorphic (as supercommutative algebras with a coaction of $\Sect(N(\M)/\M)$) to...
	\begin{enumerate}
		\item $\M = \uEnd(\A^{0|1}) \cong \A^{0|1} \rtimes \A^1$ {(pre-topological)}
		\begin{equation*}
			\pTFT{}(X) \cong \Omega^0_{R,\text{cl}}(X)
		\end{equation*}
		 closed, degree zero polynomial differential forms on $X$ over $R$.
		\item $\M = \uAut(\A^{0|1}) \cong \A^{0|1} \rtimes \G_m$ (topological)
		\begin{equation*}
			\TFT{}(X) \cong \Omega^0_{R,\text{cl}}(X)
		\end{equation*}
		 closed, degree zero polynomial differential forms on $X$ over $R$.
		\item $\M = \A^{0|1} \rtimes \Z/2\Z$  (Euclidean)
		\begin{equation*}
			\EFT{}(X) \cong \Omega^{\text{ev}}_{R,\text{cl}}(X)
		\end{equation*}
		 closed polynomial differential forms on $X$ over $R$ of even degree.		
		\item $\M  = \A^{0|1} \times 1$ (oriented Euclidean) 
			\begin{equation*}
				\EFT{}_{\text{or}}{}(X) \cong \Omega^{*}_{R,\text{cl}}(X)
			\end{equation*}
			 closed polynomial differential forms on $X$ over $R$ of arbitrary degree.
		\item $\M = 1$ (fully-rigid)
		 		\begin{equation*}
					\QFT{}_{\text{f-r}}{}(X) \cong \Omega^{*}_{R}(X)
				\end{equation*}
				all polynomial differential forms on $X$ over $R$.
	\end{enumerate}
\end{thm}

\begin{remark}
When $X$ is an infinite dimensional simplicial set we may write it as the colimit over finite dimensional skeleta and then the theorem still holds as long as $\Omega^{\text{ev}}_{R,\text{cl}}(X)$ means the product over even closed polynomial forms instead of the sum (and likewise for the last two cases).
\end{remark}

\begin{proof}
Recall that Proposition \ref{cdgastructure} gives an explicit description of the coaction of $\Sect(\underline \End(\A^{0|1}))$ on the supercommutative algebra of rational differential forms $u\Omega^{*}_{R}(X)$. In all of the cases above we compute the coinvariants for the respective coaction. For all of the following, let $a \in u\Omega^{*}_{R}(X)$.

\begin{enumerate}

\item Let $a$ be a $k$-form, then
\[
a \mapsto ax^k+(da)x^k\e.
\]
To be coinvariant $k=0$ and $da =0$.

\item This follows from 1.

\item Let $a$ be a $k$-form, then 
\[
a \mapsto a (1,-1)^k + (da)(1,-1)^k\e.
\]
To be coinvariant, $k \in 2\Z$ and $da =0 $.

\item Let $a$ be any form, then
\[
a \mapsto a+(da)\e.
\]
To be coinvariant, $da=0$.

\item This is Corollary \ref{maincor}.
\end{enumerate}

\end{proof}

\section{Twisted Field Theories} \label{sec:twists}

In Section~\ref{sec:SQFT2} we saw how a $0|1$-dimensional supersymmetric $\M$-quantum field theory assigned to each $S$-family of $\M$-bordisms over $X$ a function on $S \in \sA$. A function on $S$ is a map from $S$ to $\cO$, a section of the trivial $\cO$-line bundle over $S$. A {\em twisted} field theory is similar, except that we allow the $\cO$-line bundles to be non-trivial. 

Following \cite[\S5]{Stolz-Teichner-Survery2} a {\em twisted field theory} is defined to be a natural transformation between certain functors, the {\em twist functors}. Moreover these twist functors are functors of symmetric monoidal categories internal to \sacss, and this natural transformation is a transformation in the internal sense. The target symmetric monoidal category, $\PPic_\cO$, was introduced in Section \ref{sec:scommalg-in-sacs}. The source category is an enhancement of $\Bord_{(\M, X)}^{0|1}$ to be an internal symmetric monoidal category in \sacss.

\subsection{The bordism category} In Section~\ref{sec:SQFT2} we introduced the bordism category $\Bord_{(\M, X)}^{0|1}$ as a commutative monoid internal to \sacss. We will now promote this to a category internal to \sacss, which we denote by $\BBord_{(\M, X)}^{0|1}$ to distinguish it from our previous definition. 

For $S \in \sA$, the $S$-points of $\Bord_{(\M, X)}^{0|1}$ consisted of the equivalence classes of $S$-families of $0|1$-dimensional bordisms equipped with $\M$-structures and maps to $X$. The equivalence relation was determined by  $S$-families of $\M$-isometries. A similar description applies to $\BBord_{(\M, X)}^{0|1}$ only now instead of forming the quotient by the $\M$-isometries, the $\M$-isometries form the morphisms between the objects of $\BBord_{(\M, X)}^{0|1}(S)$. Furthermore, $X$ is now permitted to be any \sacp. 

If $Z$ is a \sacs~and $M$ is a superalgebraic cartesian monoid acting on $Z$, let $Z \mmod M$ denote the {\em action category}, internal to \sacss, whose objects are $Z$ and morphisms are $Z \times M$. The source and target maps are given by projection and the action and the composition in given by the monoid structure of $M$. In complete analogy to Proposition~\ref{prop:Bord_as_free_comm_monoid} we have:

\begin{prop}
	The category internal to \sacss, $\BBord_{(\M, X)}^{0|1}$, is given by 
	\begin{align*}
		\BBord_{(\M, X)}^{0|1} &\cong  \coprod_{k \in \N} \left( \prod_k \usCart(\A^{0|1}, X) \mmod \M \wr \Sigma_k \right) 
	\end{align*}
	In short $\BBord_{(\M, X)}^{0|1}$ is the free symmetric monoidal category, internal to \sacss, generated by the category $\usCart(\A^{0|1}, X) \mmod \M$.
\end{prop}

\subsection{Twisted field theories}
\begin{definition}\label{def:twists}
	A {\em twist} for $0|1$-dimensional supersymmetric $\M$-quantum field theories over $X$ is a functor
	\begin{equation*}
		\tau: \BBord_{(\M, X)}^{0|1} \to \PPic_\cO
	\end{equation*}
	of internal symmetric monoidal categories. 
\end{definition}

\noindent The twists, together with the internal natural transformations of twists, form a symmetric monoidal category. In particular we can take the tensor product of twists. Moreover this symmetric monoidal category is contravariantly functorial in $X$ (since the bordism category is covariantly functorial in $X$). 

\begin{example}
	The {\em trivial twist} $\tau_0: \BBord_{(\M, X)}^{0|1} \to \PPic_\cO$ is the constant symmetric monoidal functor with value the unit object of $\PPic_\cO$. This is the unit object in the symmetric monoidal category of twists. 
\end{example}

\begin{definition}\label{def:twisted-field-Thy}
	Let $\tau$ be a twist for $0|1$-dimensional supersymmetric $\M$-quantum field theories over $X$. A {\em $\tau$-twisted field theory} is an internal natural transformation $Z$ between internal functors:
\begin{center}
\begin{tikzpicture}
		\node (L) at (0, 0) {$\BBord_{(\M, X)}^{0|1}$};
		\node (R) at (4, 0) {$\PPic_\cO$};
		\node  at (2.3, 0) {$\Downarrow Z$};
		\draw [->, bend left] (L) to  node [above] {$\tau_0$} (R);
		\draw [->, bend right] (L) to  node [below] {$\tau$} (R);
\end{tikzpicture}
\end{center}
from the trivial twist $\tau_0$ to the twist $\tau$.
\end{definition}

\begin{example}[c.f. {\cite[Lma 5.7]{Stolz-Teichner-Survery2}}]
	A $\tau_0$-twisted field theory is an internal natural endo-transformation of the constant twist. Since $\cO$ may be identified with the endomorphisms of the unit object of $\PPic_\cO$, such natural transformations amounts to an $\M$-invariant $\cO$-valued function on $\usCart(\A^{0|1}, X)$. Hence $\tau_0$-twisted field theories are precisely the quantum field theories from Section~\ref{sec:SQFT2}.	
\end{example}

Since $\BBord_{(\M, X)}^{0|1}$ is free, a twist is determined by an (internal) functor
\[
\bar{\tau}: \usCart(\A^{0|1}, X) \mmod \M \lra{} \PPic_\cO.
\]
General twists over a general space $X$ can be quite interesting (see \cite{SST} for computation of general twists in the related context of $0|1$-field theories in the sense of HKST, that is in the context of supermanifolds rather than \sacss). 

In the remainder of this work we will only consider the simplest twists, which are pulled back from the case $X = pt$. We will call such twists {\em basic}. They are easy to classify:

\begin{lemma} \label{basics}
	The basic twists are classified (up to isomorphism) by
	\begin{enumerate}
		\item an object $\cL \in \ob \PPic_\cO = \ob \Pic_R$, and 
		\item a representation $\rho:\M \to \uHom(\cL, \cL) \cong \cO$.
	\end{enumerate}
	The tensor product of basic twists tensors these two pieces of data.  
\end{lemma}

\begin{proof}
	When $X = pt$ we have $\usCart(\A^{0|1}, X) \simeq pt$, and hence a basic twist is the same as an internal functor $pt \mmod \M \to \PPic_\cO$. Such a functor consists of exactly the claimed data. 
\end{proof}

\begin{thm} \label{thm:twists}
	If $X$ is a finite dimensional simplicial set, the $(\cL, \rho)$-twisted field theories over $X$ are in bijection with the $\cO(\M)$-coinvariants of $\Omega^*_R(X; \cL)$.
\end{thm}
\begin{proof}
Recall that $\cL$ may be viewed as an invertible $\cO$-module or as an invertible $R$-module. Recall the isomorphisms
\[
\Omega^*_R(X; \cL) \cong \Omega^{*}_{R}(X)\otimes_R \cL \cong \Omega(X;\cL).
\]
The natural coaction of $\cO(\M)$ on $\cO(\usCart(\A^{0|1}, X))$ from Proposition \ref{prop:finitedim} extends to a coaction on $\Omega^*_R(X; \cL)$ by tensoring up to $\cL$ and using the isomorphisms above.

Let $\tau$ be a basic twist. A natural transformation $\tau_0 \Rightarrow \tau$ of symmetric monoidal functors (internal to $\sCart$) is determined by a natural transformation of functors (internal to $\sCart$) $\bar{\tau_0} \Rightarrow \bar{\tau}$. Since $\tau$ is basic, Lemma \ref{basics} implies that it is determined by an object $\cL \in \ob \PPic_\cO$ and a representation $\rho$ of $\M$. Thus an internal natural transformation is determined by a map of superalgebraic cartesian sets
\[
\usCart(\A^{0|1}, X) \lra{n} \uHom_{\cO}(\cO,\cL) \cong \cL
\]
such that the following diagram commutes:
\begin{center}
\begin{tikzpicture}
	\node (LT) at (0, 1.5) {$\M \times \usCart(\A^{0|1}, X)$};
	\node (LB) at (0, 0) {$\usCart(\A^{0|1}, X)$};
	\node (RT) at (6, 1.5) {$\uHom_{\cO}(\cL,\cL)\times\uHom_{\cO}(\cO,\cL)$};
	\node (RB) at (6, 0) {$\uHom_{\cO}(\cO,\cL)$,};
	\draw [->] (LT) -- node [left] {$t$} (LB);
	\draw [->] (LT) -- node [above] {$\rho \times n$} (RT);
	\draw [->] (RT) -- node [right] {$t$} (RB);
	\draw [->] (LB) -- node [below] {$n$} (RB);
\end{tikzpicture}
\end{center}
where $t$ is the action map. But the action of $\uHom_{\cO}(\cL,\cL)$ on $\uHom_{\cO}(\cO,\cL)$ is the action of $\cO$ on $\cL$, so the square becomes
\begin{center}
\begin{tikzpicture}
	\node (LT) at (0, 1.5) {$\M \times \usCart(\A^{0|1}, X)$};
	\node (LB) at (0, 0) {$\usCart(\A^{0|1}, X)$};
	\node (RT) at (6, 1.5) {$\cO \times \cL$};
	\node (RB) at (6, 0) {$\cL$.};
	\draw [->] (LT) -- node [left] {$t$} (LB);
	\draw [->] (LT) -- node [above] {$\rho \times n$} (RT);
	\draw [->] (RT) -- node [right] {$t$} (RB);
	\draw [->] (LB) -- node [below] {$n$} (RB);
\end{tikzpicture}
\end{center}
But this exactly means that $n$ is $\rho$-coinvariant for the coaction discussed above. 
\end{proof}

\begin{remark}
When $X$ is infinite dimensional then we consider it as a colimit of its finite skeleta. The theorem provides a bijection for each finite skeleton and thus a bijection in the inverse limit. 
\end{remark}

\subsection{Calculation of all twists} We will first consider the case of pre-topological structures $\M \cong \A^{0|1}\rtimes \A^1$. We need to calculate all of the actions of $\A^{0|1}\rtimes \A^1$ on $\A^{1|1}$. 
Such an action consists of a map 
\begin{equation*}
	\mu: \M \times \cO \to \cO
\end{equation*}
which is unital and satisfies three properties:
\begin{enumerate}
	\item it is associative with respect to the multiplication of $\M$;
	\item it is distributes over the addition of $\cO$;
	\item it commutes with the scalar multiplication of $\cO$ on $\cO$. 
\end{enumerate}
This is the same as a unital function: 
\[
R[y, \e] \lra{\mu^*} R[y,\e,x,\de]
\]
satisfying the three conditions.

An arbitrary map $\mu^*$ is described by a pair of values:
\[
y \mapsto f_0(x,y)+f_1(x,y)\de \e
\]
and
\[
\e \mapsto g_0(x,y)\e + g_1(x,y)\de.
\]
and each of the commutative squares give restrictions on the allowed functions $f_0$, $f_1$, $g_0$, and $g_1$. 

The first condition gives the commutative square: 
\[
\xymatrix{R[y,\e,x,\de] \ar[r]^-{\mu^* \otimes 1} & R[y,\e,x_1,\de_1,x_2,\de_2] \\ R[y,\e] \ar[u]^-{\mu^*} \ar[r]^-{\mu^*} & R[y,\e,x,\de]. \ar[u]_-{1\otimes m^*}}
\]
\noindent Going around the diagram the two ways for $y$ gives
\begin{align} \label{eq1}
	& f_0(x_1x_2,y) +\e(f_1(x_1x_2,y)x_1\de_2+f_1(x_1x_2,y)\de_1)  \\
	& = f_0(x_2,f_0(x_1,y)+\e f_1(x_1,y)\de_1)+  \nonumber\\ 
	&(g_0(x_1,y)\e + g_1(x_1,y)\de_1)(f_1(x_2,f_0(x_1,y)+\e f_1(x_1,y)\de_1))\de_2. \nonumber
\end{align}
While for $\e$ we get
\begin{align} \label{eq2}
	& g_0(x_1x_2,y)\e+g_1(x_1x_2,y)(x_1\de_2+\de_1) \\ 
	& = g_0(x_2,f_0(x_1,y)+f_1(x_1,y)\e \de_1)(g_0(x_1,y)\e  +g_1(x_1,y)\de_1) \nonumber \\
	& +g_1(x_2,f_0(x_1,y)+f_1(x_1,y)\e \de_1)\de_2. \nonumber
\end{align}
This puts strong restrictions on the possible $\mu^*$.


Similarly, each of the other two conditions put restrictions on $\mu^*$. Compatibility with respect to the additive structure of $\cO$ forces $f_0$, $f_1$, $g_0$, and $g_1$ to have the following form:
\begin{align*}
	f_0(x,y) &= p(x) y \\
	f_1(x,y) &= q(x) \\
	g_0(x,y) &= r(x) \\
	g_1(x,y) &= s(x)y. 
\end{align*}
While further requiring $\mu^*$ to commute with the scalar $\cO$-multiplication forces
\begin{equation*}
	p(x) = r(x) \quad \text{ and } \quad q(x) = s(x). 
\end{equation*}

Returning to Equations (\ref{eq1}) and (\ref{eq2}), we see that $\mu^*$ defines a unital $\M$-action on the $\cO$-module $\cO$ if and only if
\begin{equation*}
	q(x) = 0, \quad \quad  p(x_1 x_2) = p(x_1) p(x_2),  \text{ and } \quad \quad p(1) = 1. 
\end{equation*}

%

\begin{lemma} \label{polylemma}
Let $R$ be a connected ring of characteristic $0$ and $p(x) \in R[x,x^{-1}]$. If
\[
p(x_1x_2) = p(x_1)p(x_2)
\]
then either $p(x) = 0$ or $p(x) = x^n$ for some $n$.
\end{lemma}
\begin{proof}
This implies that
\[
p(x)^2 = p(x^2).
\]
Since $R$ is connected we have that $p(1) = 1$ or $p(1) = 0$. 

Assume that $p(1) = 0$. Then we see immediately that
\[
p(x) = 0.
\]

Assume that $p(1) = 1$. Let
\[
w(x) = x^mp(x),
\]
then
\[
w(x_1)w(x_2) = w(x_1x_2).
\]
By choosing $m$ large enough, $w(x)$ has the form
\[
rx^n + \text{lower degree terms},
\]
where $n>0$, $r \neq 0$, and the lower degree terms are in positive degree.
The equality $w(x_1)w(x_2) = w(x_1x_2)$ implies that $r^2 = r$, so $r=1$. Now we may take the derivative with respect to $x_1$ $n$-times in order to get the equality
\[
w^{(n)}(x_1x_2)x_{2}^n = w(x_2)w^{(n)}(x_1),
\] 
but this is just the statement that
\[
(n !) \cdot w(x_2) =  x_{2}^n \cdot (n !), 
\]
which is equivalent to $w(x) = x^n$ since we are in characteristic zero. 
This implies the result for $p(x)$ by the definition of $w(x)$.
\end{proof}

This completely determines the basics twists for the pre-topological geometry in characteristic zero. If $R$ is connected, then for each $n \in \N$ there is a \emph{degree $n$ twist}
\begin{align*}
	y &\mapsto x^n y \\
	\epsilon & \mapsto x^n \epsilon
\end{align*}

A similar calculation determines the possible basic twists for the remaining geometries. Below is a table containing the forms of the $\cO$-linear actions of $\M$ on $\A^{1|1}$ for each of the geometries $\M = \A^{0|1}\rtimes \Z/2$ and $\M=\A^{0|1}$. 

\begin{table}[h] 
\begin{center}
{	\small
\begin{tabular}{|l|l|}
\hline
Geometry & Coaction \\
\hline
\hline
$\A^{0|1} \rtimes \Z/2$ & $y \mapsto y$ \\ 
$R[x,\de]/(x^2-1,\de^2)$ & $\e \mapsto \e $\\ \cline{2-2} 
& $y \mapsto xy$ \\
& $\e \mapsto x \e$ \\
\hline
$\A^{0|1}$ & $y \mapsto y$ \\
& $\e \mapsto \e$ \\
\hline
\end{tabular}
}
\end{center}
\end{table}

We will refer to these as the degree $n$ twists, where in the pre-topological geometry $n \in \N$, in the topological geometry $n \in \Z$, in the Euclidean geometry $n \in \{ \text{odd, even} \}$, and in the oriented Euclidean geometry there is only one twist. 

Translating this to twisted field theories yields:

\begin{thm}\label{thm:twisted-field-theories}
	Over a ring $R$ of characteristic zero, for each geometry and each degree $n$ basic twist, the twisted field theories are given in Table~\ref{table:degree_twisted_Field-Theories}. 
	In the case of the pretopological geometry these pick out closed forms $\omega \in \Omega^k$ (with even super grading) and $\alpha \in \Omega^k$ (with odd super grading), $k\in \mathbb{N}$. When $\Omega^* = \Omega^*_R(X,\cL)$ is the cdga of differential forms on a simplicial set $X$, there are only forms of odd super grading in odd degrees and only forms of even super grading in even degrees. Thus fixing $k$ picks out precisely the forms of degree $k$ (one of $\omega$ or $\alpha$ must be zero). 

	The topological geometry behaves in precisely the same way and, of course, there are only forms in nonnegative degrees when the cdga consists of the forms on a simplicial set. 

	In the case of the Euclidean geometry taking both $\omega$ and $\alpha$ to be even gives the even forms on $X$ and taking both to be odd gives the odd forms on $X$.
\end{thm}

We will denote the collection of degree $n$ twisted field theories with a geometry by a superscript $n$. Thus $\pTFT{}^{n}(X)$ denotes the degree $n$  pretopological field theories.

\begin{table}[pthb] 
\begin{center}
{	\small
\begin{tabular}{|l|l|}
\hline
Geometries & Twisted Field Theories \\
\hline
\hline
$\A^{0|1}\rtimes \A^{1}$ & $\omega \in \Omega^k$, $\alpha \in \Omega^k$ \\ Pretopological & $d\omega = 0$, $d\alpha = 0$, $k \in \N$ \\ 
\hline
$\A^{0|1} \rtimes \G_m$  & $\omega \in \Omega^k$, $\alpha \in \Omega^k$ \\ Topological & $d\omega = 0$, $d\alpha = 0$, $k \in \Z$ \\ 
\hline
$\A^{0|1} \rtimes \Z/2$ & $\omega \in \Omega^{\text{even or odd}}$, $\alpha \in \Omega^{\text{even or odd}}$ \\ Euclidean & $d\omega = 0$, $d\alpha = 0$ \\ 
\hline
$\A^{0|1}$ & $\omega \in \Omega^*$, $\al \in \Omega^*$ \\ oriented Euclidean & $d\omega = 0$, $d\alpha = 0$ \\ 
\hline
\end{tabular} \\[2pt]
Here the super grading of $\omega$ is even, while that of $\alpha$ is odd.
}
\end{center}
	\caption{General form of basic twists}
	\label{table:degree_twisted_Field-Theories}
\end{table}

\section{Concordance} \label{sec:concordance}


Theorems~\ref{thm:field-theories} and~\ref{thm:twisted-field-theories} and Table~\ref{table:degree_twisted_Field-Theories} show that the twisted superalgebraic cartesian quantum field theories with a geometry over a simplicial set $X$ correspond to important subsets of Sullivan's rational differential forms on $X$. To recover the rational cohomology groups of $X$ from the field theories we study a notion of equivalence of field theories called concordance. In this algebraic setting we uncover three notions of concordance. We prove that they are all equivalent. In each case two closed differential forms are concordant if and only if they are cohomologous. 

Given a simplicial set $X$, we may consider the two inclusions
\[
f_0,f_1:X \lra{} X \times \Delta^1
\]
induced by the coface maps of $\Delta^1$. Now, using the the canonical map 
\[
i_!(X \times \Delta^1) \lra{} i_!X \times \A^1
\]
and the canonical map $\Omega^*(i_!X) \otimes \Omega^*(\A^1) \lra{} \Omega^*(i_!X \times \A^1)$ we build the commutative diagram:
\[
\xymatrix{ & \Omega^*(i_!X) \otimes \Omega^*(\A^1) \ar[d] \ar@/^2pc/[dddr] \ar@/_2pc/[dddl] & \\ & \Omega^*(i_!X \times \A^1) \ar[d] & \\ & \Omega^{*}(i_!(X \times \Delta^1)) \ar[dr]^{f_0} \ar[dl]_{f_1} & \\ \Omega^{*}(i_!X) && \Omega^{*}(i_!X).}
\]

Note that the downward arrows need not be isomorphisms. We use this diagram to describe the three notions of concordance for two differential forms $\omega_0, \omega_1 \in \Omega^{*}_{\text{cl}}(i_!X)$. They fit nicely into a table:
\begin{center}
\begin{tabular}{|l|l|}
\hline
Cohomologous & $\exists \al, \text{ }\omega_0 - \omega_1 = d\alpha$  \\
\hline
Cochain Concordance & $\exists \omega \in \Omega^{*}_{\text{cl}}(i_!X) \otimes \Omega^{*}_{\text{cl}}(\A^1), \text{ } f_j \omega = \omega_j$\\
\hline
Algebraic Concordance & $\exists \omega \in \Omega^{*}_{\text{cl}}(i_!X \times \A^1), \text{ } f_j \omega = \omega_j$ \\
\hline
Simplicial Concordance & $\exists \omega \in \Omega^{*}_{\text{cl}}(i_!(X \times \Delta^1)), \text{ } f_j \omega = \omega_j$ \\
\hline
\end{tabular}
\end{center}

It is immediate that Cochain Concordance implies Algebraic Concordance implies Simplicial Concordance.

\begin{prop}
Cohomologous implies Cochain Concordance.
\end{prop}
\begin{proof}
The element $\omega_1 t + \omega_0(1-t) + \al dt \in \Omega^{*}_{\text{cl}}(i_!X) \otimes \Omega^{*}_{\text{cl}}(\A^1)$ does the job.
\end{proof}

\begin{prop}
Let $R$ be a $\Q$-algebra, then Simplicial Concordance implies Cohomologous.
\end{prop}
\begin{proof}
It suffices to take $\omega \in \Omega^{*}(i_!(X \times \Delta^1))$ such that $f_0 \omega = 0$ and $f_1 \omega = \omega_1$. We must show that there exists $\al$ such that $d\alpha = \omega_1$. However, because $X \times \Delta^1 \simeq X$, by Sullivan's theorem \cite{Sull} $f_0$ and $f_1$ are  quasi-isomorphisms. Thus the cohomology class of $\omega_1$ equals the cohomology class of $0$.
\end{proof}


We use square brackets to denote the set of twisted field theories with a given geometry taken up to concordance. Thus $\pTFT{}^{n}[X]$ denotes the degree $n$ pretopological field theories up to concordance. 

\begin{thm}
Let $R$ be a $\Q$-algebra, $HR$ be cohomology with coefficients in $R$, and $X$ be a simplicial set. There are natural isomorphisms
\[
\pTFT{}^n[X] \cong \TFT{}^n[X] \cong HR^n(X)
\]
and
\[
\EFT{}^n[X] \cong PHR^n(X),
\]
where $PHR$ is periodic cohomology with coefficients in $R$.
\end{thm}


\begin{remark}
Because periodic cohomology is defined using the \emph{product}, $X$ may be taken to be an infinite dimensional simplicial set.
\end{remark}

\section{A variation for cohomology over any ring} \label{sec:cohomology}

In this final section we will describe a variant of the above results which allows one to recover the  cohomology of a simplicial set \emph{over any ring} as concordance classes of degree $n$-field theories over that simplicial set.  

In the world of superalgebraic Cartesian sets (over $R$) the functions on the space $\A^{n|q}$ are the supercommutative ring $R[x_1,\ldots,x_n,\e_1,\ldots,\e_q]$. In short the functions on affine space are only those functions on $(n|q)$-variables obtainable from standard algebraic operations for rings. In the closely related context of supermanifolds the functions on $\R^{n|q}$ are given by $C^\infty(\R^n)[\e_1, \ldots, \e_q]$. This corresponds to functions on $(n|q)$-variables obtainable using $C^\infty$-operations (see \cite{lawvere1979categorical,moerdijk2013models} for discussions of the algebraic theory of $C^\infty$-rings). We will now describe an intermediate theory which includes not only standard algebraic operations for rings but also \emph{divided power operations}. 

Let $R$ be a ring and fix an ideal $I \subseteq R$. Recall that a \emph{divided powers structure on $(R, I)$} is a collection of maps $\gamma_n: I \to R$ for $n = 0, 1, 2, \ldots$ such that
\begin{enumerate}
	\item $\gamma_0 =1$ and $\gamma_1(x) = x$ for $x \in I$, while $\gamma_n(x) \in I$ for $n > 0$;
	\item $\gamma_n(x + y) = \sum_{i= 0}^n \gamma_{n-i}(x) \gamma_i(y)$ for $x, y \in I$;
	\item $\gamma_n(\lambda x) = \lambda^n \gamma_n(x)$ for $\lambda \in R$, $x \in I$;
	\item $\gamma_m(x) \gamma_n(x) = {m+n \choose n} \gamma_{m + n}(x)$ for $x \in I$;
	\item $\gamma_n(\gamma_m(x)) = \frac{(mn)!}{(m!)^n n!} \gamma_{mn}(x)$.
\end{enumerate}
A \emph{homomorphism} of rings with divided powers structure $(R, I) \to (R', I')$ is a ring homomorphism $R \to R'$ sending $I$ into $I'$ and commuting with the maps $\gamma_n$. 

Now fix a ring $S$. On a first reading the reader might wish to focus on the case $S = \Z$.
The base ring that our spaces will be defined over is not $S$ but the ring $R = \Gamma_S(t)$, the free divided powers algebra over $S$ on one variable $t$. As an $S$-module $R$ is free with basis $1$, $\gamma^r(t)$, $r\geq 1$ and has multiplication:
\begin{equation*}
	\gamma^r(t) \cdot \gamma^s(t) = {r+s \choose r} \gamma^{r+s}(t).
\end{equation*}
The element $\gamma^r(t)$ should be thought of as the formal expression `$\frac{t^r}{r!}$'. When $S = \Z$
this ring may be defined as the smallest subring of $\Q[t]$ containing $\Z$ and each of the monomials $\frac{t^n}{n!}$. Note that $R$ is flat over $S$. 

In our new theory of spaces over $R$, the functions on the affine spaces $\A^{n|q}$ are given by the divided powers algebra
\begin{equation*}
	\cO(\A^{n|q}) = \Gamma_S(t, x_1, \ldots, x_n)[\e_1, \ldots, \e_q]
\end{equation*}
that is we take the free divided powers algebra on the variables $t, x_1, \ldots, x_n$ and tensor with the exterior algebra on the variables $\e_1, \dots, \e_q$. This is naturally a divided powers algebra where $\gamma_n(\e_j) = 0$ for $n > 1$. The morphisms from $\A^{m|p}$ to $\A^{n|q}$ are defined to be the $R$-algebra homomorphisms
\begin{equation*}
	\cO(\A^{n|q}) \to \cO(\A^{m|p})
\end{equation*}
which are also homomorphisms of divided powers algebras. Using these $\A^{m|p}$ and morphism, the category of superalgebraic Cartesian sets (now with divided power operations) may be constructed just as before. 
In particular we still have
\begin{equation*}
	\A^{m|p} \times \A^{n|q} = \A^{m + n | p + q}.
\end{equation*}
One difference is that the functor $\cO(-)$ is no longer represented by the object $\A^{1|1}$. If we denote the ideal supporting the divided powers structure on $\cO(\A^{m|p})$ by
\begin{equation*}
	I\cO(\A^{m|p}) = \left( \gamma_n(t), \gamma_n(x_i), \e_j  \; | \; n \geq 1, 1 \leq i \leq m, 1 \leq j \leq p \right)
\end{equation*}
then we have:
\begin{align*}
	\sA( \A^{m|p}, \A^{1|0}) &\cong I\cO(\A^{m|p})_{\text{ev}} \\
	\sA( \A^{m|p}, \A^{0|1}) &\cong I\cO(\A^{m|p})_{\text{odd}}.
\end{align*}
In light of this, Lemma~\ref{lma:affine} must be replaced with:

\begin{lemma}
	We have natural isomorphisms $\underline{\sCart}(\A^{0|1}, \A^{1|0}) \cong \A^{1|1}$ and $\underline{\sCart}(\A^{0|1}, \A^{0|1}) \cong \cO$.
\end{lemma}

\begin{proof}
	We have
	\begin{align*}
	\sCart(\A^{m|p},\underline{\sCart}(\A^{0|1},\A^{1|0})) &\cong \sCart(\A^{m|p}\times \A^{0|1}, \A^{1|0}) \\ &\cong \Sect(\A^{m|p+1})_{\text{ev}} \\ &\cong I\Sect(\A^{m|p}) = \A^{1|1}(\A^{m|p}).
	\end{align*}
	and
	\begin{align*}
	\sCart(\A^{m|p},\underline{\sCart}(\A^{0|1},\A^{0|1})) &\cong \sCart(\A^{m|p}\times \A^{0|1}, \A^{0|1}) \\ &\cong I\Sect(\A^{m|p+1})_{\text{odd}} \\& \cong \Sect(\A^{m|p}). \qedhere
	\end{align*}
\end{proof}

\noindent The object $\A^{0|1}$ remains cartesian tiny, and so Cor.~\ref{cor:right-adjoint-to-innerhom}, which defines the functor $\Omega_{(-)}$, is still valid.

Similarly to before we also have an isomorphism of algebras
\begin{equation*}
	\cO(\A^{n|0}) \cong \Gamma_S(t, x_0, x_1, \ldots, x_n) / (t - \sum_i x_i)
\end{equation*}
which makes it clear that the assignment $[n] \mapsto \A^{n|0}$ defines a cosimplicial object. Using this cosimplicial object, every simplicial set gives rise to a  superalgebraic Cartesian set  with divided power operations. We have $\Omega(\A^{n|0}) = \Gamma_S(t, x_1, \ldots, x_n)[ \e_1, \cdots, \e_n]$, but we may think of this simplicially as the ring
\begin{equation*}
	\Omega(\A^{n|0}) \cong \Gamma_S(t, x_0 , x_1, \ldots, x_n)[ dx_0, \cdots, dx_n] / (t - \sum_i x_i, \sum_i dx_i).
\end{equation*}
This is similar to Sullivan's ring of polynomial differential forms on the $n$-simplex but also includes the divided powers $x_i^k / k!$. As before if $X$ is a simplicial set, then $\Omega(X)$ is a commutative dga which we will call the \emph{$S$-divided power differential forms on $X$}. We will let $\Omega_{\text{cl}}(X)$ be the closed forms, the kernel of the differential $d$. 

It is worth pointing out that $\Omega(\A^{0|q}) \ncong \cO(\A^{q|q})$, that is, Corollary 3.5 does not hold for arbitrary superalgebraic Cartesians sets. However, the formulas above imply that it does hold for superalgebraic Cartesians sets of the form $i_{!}X$, where $X$ is a simplicial set. 

The arguments from Sections~\ref{sec:SQFT2} and~\ref{sec:twists} carry over identically and yield the following theorem, which we only state for topological and Euclidean geometries for brevity. Analogous results hold for all geometries. 

\begin{thm}
	Let $X$ be a simplicial set. Then the set of supersymmetric $0|1$-dimensional degree $n$ $\M$-quantum field theories over $i_! X$ is naturally isomorphic (as supercommutative algebras with a coaction of $\Sect(N(\M)/\M)$) to...
	\begin{enumerate}
		\item $\M = \uAut(\A^{0|1}) \cong \A^{0|1} \rtimes \G_m$ (topological)
		\begin{equation*}
			\TFT{}^n(X) \cong \Omega^n_{\text{cl}}(X)
		\end{equation*}
		 closed, degree $n$ divided powers differential forms on $X$ over $R$.
		\item $\M = \A^{0|1} \rtimes \Z/2\Z$  (Euclidean)
		\begin{equation*}
			\EFT{}^n(X) \cong \begin{cases}
				\Omega^{\text{ev}}_{\text{cl}}(X) & n \text{ even} \\
				\Omega^{\text{odd}}_{\text{cl}}(X) & n \text{ odd}	
			\end{cases}
		\end{equation*}
		 closed divided powers differential forms on $X$ over $R$ of the specified parity.		
	\end{enumerate} \qed
\end{thm}

\noindent Unlike in the previous case, concordance classes of field theories \emph{are not} simply in bijection with cohomology classes in $H^*(X; \Gamma_S(t))$. Indeed it is not a priori clear that the four different notions of concordance agree in the current setup. What is clear are the implications
\begin{align*}
	\text{Cohomologous} &\Rightarrow \text{Cochain Concordance} \\
	&\Rightarrow \text{Algebraic Concordance} \\
	&\Rightarrow  \text{Simplicial Concordance.}
\end{align*} 
In fact, as a consequence of work of Cartan \cite[Section~7]{Cartan1976} and Miller \cite{miller1978derham}
all four notions of concordance \emph{do agree} and we may identify concordance classes of field theories with a specific subalgebra of $H^*(X; \Gamma_S(t))$. Our treatment will follow Cartan. 

As we have just observed, the $S$-divided power differential forms on $X$ correspond to supersymmetric $0|1$-dimensional quantum field theories over $X$, but they have another useful description.  For each $n$ we can consider the chain complex:  
\begin{equation*}
	\Omega^*(\Delta^n) = \Gamma_S(t, x_0 , x_1, \ldots, x_n)[ d x_0, \cdots, d x_n] / (t - \sum_i x_i, \sum_i \delta x_i).
\end{equation*}
Letting $n$ vary these assemble into a simplicial commutative dga $\Omega^*_\bullet$. Here $*$ denotes the differential grading while $\bullet$ denotes the simplicial degree. For a simplicial set $X$ the  $S$-divided power differential forms on $X$ are given by the complex of simplicial mapping spaces $\sSet(X_\bullet, \Omega^*_\bullet)$:
\begin{equation*}
	\sSet(X_\bullet, \Omega^0_\bullet) \stackrel{d}{\to}\sSet(X_\bullet, \Omega^1_\bullet) \stackrel{d}{\to} \cdots.
\end{equation*}
The cohomology of this complex is the cohomology of $\Omega^*(X)$. 

Following Cartan we may consider a generalization of this situation. Suppose that $A_\bullet^*$ is a simplicial dga (which we no-longer assume to be commutative). Then for any simplicial set $X$ we may form a dga $A^*(X)$ via the assignment
\begin{equation*}
	A^q(X) = \sSet(X_\bullet, A^q_\bullet).
\end{equation*} 
The cohomology $H^*(A^*(X))$ is defined to be the cohomology of this dga. Given such an $A_\bullet^*$, we let $ZA^q_\bullet$ denote the kernel of 
\begin{equation*}
	\delta: A^q_\bullet \to A^{q+1}_\bullet
\end{equation*}

Cartan considers simplicial dga's $A_\bullet^*$ which satisfy two key properties:
\begin{enumerate}
	\item[(A)] There is a long exact sequence
	\begin{equation*}
		0 \to ZA^0_\bullet \to A^0_\bullet \stackrel{\delta}{\to} A^1_\bullet \stackrel{\delta}{\to} A^2_\bullet \stackrel{\delta}{\to} \cdots
	\end{equation*}
	and $ZA^0_\bullet$ is a discrete simplicial set (i.e. it is constant); and
	\item[(B)] The simplicial homotopy group $\pi_p( A^q_\bullet) = 0$ is null if $p \neq q$ and we have a surjection
	\begin{equation*}
		\pi_p(ZA^p_\bullet) \twoheadrightarrow \pi_p(A^p_\bullet).
	\end{equation*}
\end{enumerate}
In this case we set $R = ZA^0_{[0]}$, the set of zero simplices. For each $q$ we have a short exact sequence
\begin{equation*}
	0 \to ZA^q_\bullet \to A^q_\bullet \stackrel{\delta}{\to} ZA^{q+1}_\bullet \to 0,
\end{equation*}
which is necessarily a fibration sequence of Kan complexes. The long exact sequence of homotopy groups then gives rise to a chain of inclusions
\begin{equation*}
	\cdots \hookrightarrow \pi_{q+1}(ZA^{q+1}_\bullet) \hookrightarrow \pi_{q}(ZA^{q}_\bullet) \hookrightarrow \cdots \hookrightarrow \pi_{0}(ZA^{0}_\bullet) = R. 
\end{equation*}
The image of $\pi_{q}(ZA^{q}_\bullet)$ defines the $q^\text{th}$ part of a filtration $F_qR \subseteq R$. We also observe that $ZA^q_\bullet$ is an Eilenberg-Mac Lane space (simplicial set) $K(F_qR, q)$. 

\begin{example} \label{ex:cochains}
	For any ring $R$ we have a simplicial dga $A^*_\bullet = C^*_{R, \bullet}$ given by the assignment
	\begin{equation*}
		C^q_{R,p} = C^q(\Delta^p; R)
	\end{equation*}
	of the degree $q$ simplicial cochains on $\Delta^p$ with coefficients in $R$. This satisfies both properties (A) and (B). In the case of (B) we have $\pi_n( C^n_{R, \bullet}) = 0$ so that $F_qR = R$ is the trivial filtration and $ZC^q_{R, \bullet}$ is a $K(R, q)$.
\end{example}

\begin{example} \label{ex:divdiffforms}
	Take $A^*_\bullet = \Omega^*_\bullet$, which corepresents $S$-divided power differential forms. \cite[Section~7]{Cartan1976} implies that this example satisfies properties (A) and (B). In this case $\pi_q(\Omega^q_\bullet) = S$ is non-zero and so we get a non-trivial filtration of $Z\Omega^0_{[0]} = R = \Gamma_S(t)$, the free divided powers $S$-algebra on the generator $t$. We have \cite[Section~7]{Cartan1976} that 
	\begin{equation*}
		F_q  \Gamma_S(t) =  \Gamma_S^{\geq q}(t)
	\end{equation*}
	consists of the divided powers of \emph{weight} at least $q$ (i.e. the $S$-submodule spanned by $\gamma^r(t)$ for $r \geq q$). 
\end{example}

\begin{thm}[Cartan-Miller \cite{Cartan1976,miller1978derham}]
	Let $A_\bullet^*$ be a simplicial dga satisfying properties (A) and (B) above and let $R = ZA^0_{[0]}$ and the filtration $F_qR \subseteq R$ be as above. Suppose that in addition $A_\bullet^*$ satisfies 
	\begin{enumerate}
		\item[(C)] $A^q_p = 0$ for $q > p$. 
	\end{enumerate}
	Then there is a unique \emph{integration map} of simplicial chain complexes 
	\begin{equation*}
		I: A_\bullet^* \to C^*_{R, \bullet}
	\end{equation*}
	which restricts to the identity $R = ZA^0_{[0]} \to ZC^0_{R,[0]} = R$. Moreover we have natural isomorphisms $H^q( A^*(X)) \cong [X_\bullet, ZA^q_\bullet] \cong H^q(X; F_qR)$ (where the middle term represents simplicial homotopy classes) and under this isomorphism $I$ agrees with the map induced by the inclusion $F_q(R) \subseteq R$. 
	
	Moreover, suppose that $R$ is a flat $k$-algebra (over some ring $k$) and that $ZA^q_{[p]}$ is $k$-flat for all $q, p$. Then the induced map 
	\begin{equation*}
		H(I): H^*(A^*(X)) \to H^*(X; R)
	\end{equation*}
	is multiplicative (it's a homomorphism of graded $k$-algebras). \qed
\end{thm}

Property (C) is satisfied by both Examples~\ref{ex:cochains} and~\ref{ex:divdiffforms}. The isomorphism $H^q( \Omega^*(X)) \cong [X_\bullet, Z\Omega^q_\bullet]$ shows that the simplical concordance agrees with the cohomologous relation, and hence all four notions of concordance of field theory agree. Because $\Gamma_S(t)$ is a free module over $S$, the flatness condition in the last part of the theorem is also satisfied in Example~\ref{ex:divdiffforms} if we take $k = S$. We have the following immediate corollary: 

\begin{cor}
	Let $X$ be a simplicial set. Then we have a natural isomorphism 
	\begin{equation*}
		\TFT{}^n[X] \cong H^n(X; \Gamma^{\geq n}_S(t))
	\end{equation*}
	between the set of concordance classes of supersymmetric $0|1$-dimensional degree $n$ topological field theories over $X$ and the cohomology of $X$ in the $\Gamma_S(t)$-module $\Gamma^{\geq n}_S(t)$. Moreover letting $n$-vary, this isomorphism is multiplicative, where the right-hand-side is viewed as a summand of the graded ring $H^*(X; \Gamma_S(t))$. \qed
\end{cor}

\noindent Since $\Gamma_S^{\geq n}(t)$ is a flat $S$-module we have that
\begin{equation*}
	 H^*(X; \Gamma^{\geq n}_S(t)) \cong H^*(X; S) \otimes_S \Gamma_S^{\geq n}(t)
\end{equation*}
is obtained by base-change. It follows that we can recover the additive $S$-cohomology of $X$ from the concordance classes of topological field theories $\TFT{}^*[X]$. For example we can identify $H^n(X; S)$ with the summand $H^n(X; \Gamma^{n}_S(t)) \subseteq H^*(X; \Gamma^{\geq n}_S(t)) \cong \TFT{}^n[X]$
corresponding to degree $n$-forms with coefficients of weight exactly $n$. Under this identification, however, the natural pairing
\begin{equation*}
	\TFT{}^m[X] \times \TFT{}^n[X] \to \TFT{}^{m+n}[X]
\end{equation*}
sends $\alpha \in H^m(X;S)$ and $\beta \in H^n(X;S)$ to 
\begin{equation*}
	{m + n \choose n} \cdot \alpha \cup \beta \in H^{m+ n}(X;S).
\end{equation*}
In particular the natural multiplication for TFTs only sees multiples of the cup product structure on $H^*(X; S)$.  We end with a final question.

\begin{question}
	Is there a natural quantum field theoretic construction that would allow one to recover the full structure of $H^*(X;S)$ as a graded ring from the concordance classes of topological field theories $\TFT{}^*[X]$?
\end{question}

\bibliographystyle{alpha}
\bibliography{mybib}
\end{document}